\documentclass[article]{amsart}
\usepackage{amsmath,amssymb,amsthm}
\usepackage[latin1]{inputenc}
\usepackage{version,tabularx,multicol}
\usepackage{graphicx,float, url}
\usepackage{amsfonts}
\usepackage{quotes}
\usepackage{graphicx}
\usepackage{epsfig}
\usepackage{subfigure}
\usepackage{color}
\usepackage{float}
\usepackage{color}
\usepackage{flafter}

\vspace{0.3cm}

\theoremstyle{plain}
\newtheorem{theo}{Theorem}[section]

\newtheorem{cor}[theo]{Corollary}
\newtheorem{prop}[theo]{Proposition}

\newtheorem{lem}[theo]{Lemma}

\theoremstyle{remark}
\newtheorem{rem}[]{Remark}[section]

\newcommand{\ind}{{\bf 1}}

\newcommand{\Var}{{\rm Var}}

\begin{document}

\title{On the measure division construction of $\Lambda$-coalescents}

\date{\today}

\author{Linglong YUAN}
\address{University\'e Paris 13, Sorbonne Paris Cit\'e, LAGA, CNRS, UMR 7539, F-93430, Villetaneuse, France.}
\email{yuan@math.univ-paris13.fr}

\begin{abstract}
{ This paper provides a new construction of $\Lambda$-coalescents called "measure division construction". This construction is pathwise and consists of dividing the characteristic measure $\Lambda$ into several parts and adding them one by one to have a whole process. Using this construction, a "universal" normalization factor  $\mu^{(n)}$ for the randomly chosen external branch length $T^{(n)}$ has been discovered for a class of coalescents. This class of coalescents covers processes similar to Bolthausen-Sznitman coalescent,  the coalescents without proper frequencies, and also others. }
\end{abstract}


\maketitle
\section{Introduction}
\subsection{Motivation and main results}
Let $\mathbb{N}:=\{1,2,\cdots\}$, $\Omega$ be a subset of $ \mathbb{N}$ and $\pi$ a partition of $\Omega$ such that $|\pi|<+\infty$ ($|\pi|$ denotes the number of blocks in $\pi$). The $\Lambda$-coalescent process starting from $\pi$, introduced independently by Pitman\cite{MR1742892} and Sagitov\cite{MR1742154}, is denoted by $\Pi^{(\pi)}:=(\Pi^{(\pi)}(t))_{t\geq 0}$, where $\Pi^{(\pi)}(0)=\pi$ and $\Lambda$ is a finite measure on $[0,1]$. Here we specify that a finite measure on $[0,1]$ can be a null measure and hence its total mass is a non-negative real value. If $\pi=\{\{1\}, \{2\}, \cdots, \{n\}\}$, i.e., the set of first $n$ singletons, then  the process is simply denoted by $\Pi^{(n)}$. In this paper, we will frequently use two other notations $\Lambda_1, \Lambda_2$ for finite measures. We define then $\Pi^{(1,n)}$ as the $\Lambda_1$-coalescent and $\Pi^{(2,n)}$ the $\Lambda_2$-coalescent, both taking $\{\{1\}, \{2\}, \cdots, \{n\}\}$ as initial value.

This process $\Pi^{(\pi)}$ is a continuous time Markov process with c\`adl\`ag trajectories taking values in the set of partitions of $\Omega$.  More precisely: Assume that at time $t$, $\Pi^{(\pi)}(t)$ has $b$ blocks, then after a random exponential time with parameter $g_b$
\begin{equation}\label{gn}
g_b=\sum_{k=2}^{b}{ b \choose k}\lambda_{b,k},  \quad \text{where} \quad \lambda_{b,k}=\int_0^1x^{k-2}(1-x)^{b-k}\Lambda(dx),
\end{equation}
 $\Pi^{(\pi)}$ encounters a collision and the probability for a group of $k (2\leq k\leq b)$ blocks to be merged into a bigger block with the other $b-k$ blocks unchanged is 
\begin{equation*} \frac{\lambda_{b,k}}{g_b}.
\end{equation*}
Then 
\begin{equation}\label{pbk}
p_{b,b-k+1}:=\frac{{b \choose k}\lambda_{b,k}}{g_b}
\end{equation}
is the probability to have $b-k+1$ blocks right after the collision. This definition gives the exchangeability of blocks.  In particular, for any permutation $\rho$ on $\{1,2,\cdots, n\}$, $\rho\circ \Pi^{(n)}\stackrel{(d)}{=}\Pi^{(n)}.$

Remark that if $\Lambda(\{0\})=0$, then we get the following well known formula:
\begin{equation}\label{gn1}
g_b=\int_0^1(1-(1-x)^b-bx(1-x)^{b-1})x^{-2}\Lambda(dx).
\end{equation}

The definition shows that the law of $\Pi^{(\pi)}$ is determined by the initial value $\pi$ and the measure $\Lambda$ which is hence called characteristic measure. 

Notice that $\Omega$ can be an abstract set and the coalescing mechanism works all the same.  The reason why one takes $\Omega$ as a subset of $\mathbb{N}$ relies on its applications in the genealogies of populations. We take $\Pi^{(n)}$ as an example where $\Omega=\{1, 2, \cdots, n\}$.  At time $0$, we have $\Pi^{(n)}(0)=\{\{1\}, \{2\}, \cdots, \{n\}\}$ which is interpreted as a sample of $n$ individuals labelled from $1$ to $n$ . 
If at time $t$, $\Pi^{(n)}$ has its first coalescence where $\{1\}$ and $\{2\}$ are merged together with the others unchanged, then $\Pi^{(n)}(t)=\{\{1,2\}, \{3\}, \cdots, \{n\}\}$ which is interpreted as getting the MRCA (most recent common ancestor) $\{1,2\}$ of individuals $1$ and $2$ with the others unchanged at that time. Hence $\{1,2,\cdots,n\}$ is an absorption state of $\Pi^{(n)}$ and is the MRCA of all individuals. For more details, we refer to \cite{MR671034, kingman2000oc} or \cite{arnason2004mitochondrial, boom1994mdv, Eldon2006, hedgecock19942}.

 Let $1\leq m\leq n$ and $\sigma$ the restriction from $\{1,2,\cdots, n\}$ to $\{1,2,\cdots, m\}$. We have the consistency property: $\sigma$ $\circ$ $\Pi^{(n)}\stackrel{(d)}{=}\Pi^{( m)}$ (see \cite{MR1742892}). According to this property and exchangeability of blocks, if $\pi'$ is a subset of $\pi$, then the restriction of $\Pi^{(\pi)}$ from $\pi$ to $\pi'$ has the same distribution as that of $\Pi^{( \pi')}$. We can also define $\Pi^{(\pi)}$ when $|\pi|=+\infty$ by using the consistency property and the definition in finite cases (see \cite{MR1742892}).

 Let $|\Pi^{(n)}|$ be the block counting process associated to $\Pi^{(n)}$ such that $|\Pi^{(n)}(t)|$ is the number of blocks of $\Pi^{(n)}(t)$ for any $t\geq 0$. Then it decreases from $n$ at time $0$.  We denote by $X_1^{(n)}$ the decrease of number of blocks  at the first coalescence.  For $i\in\left\{ 1,\ldots,n\right\}$, we define
$$T^{(n)}_i:=\inf\left\{t\geq 0|\left\{ i\right\}\notin \Pi^{(n)}_t\right\}$$
the length of the $i$th external branch and $T^{(n)}$ the length of a randomly chosen external branch. By exchangeability, $T_i^{(n)}\stackrel{(d)}{=}T^{(n)}$. We denote by $L_{ext}^{(n)}:=\sum_{i=1}^{n}T_i^{(n)}$ the total external branch length  of $\Pi^{(n)}$, and by $L_{total}^{(n)}$ the total branch length. 

There are four classes of $\Lambda$-coalescents having been largely studied. We give the results concerning $T^{(n)}$, which show a common regularity that we will discuss later. 
\begin{itemize}
\item $\Lambda=\delta_{0}$: Kingman coalescent (see \cite{MR671034}, \cite{MR633178}). Then $\displaystyle  nT^{(n)}$ is asymptotically distributed with density function $\frac{8}{(2+x)^{3}}\ind_{x\geq 0}$ (See \cite{MR2156553}, \cite{caliebe2007length}, \cite{janson2011total}). 
\item $\Lambda=\Lambda^{leb}$: Bolthausen-Sznitman coalescent (see \cite{bolthausen1998ruelle}).  Here $\Lambda^{leb}$ denotes the Lebesgue measure on $[0,1]$.  Then $\displaystyle  (\ln n)T^{(n)}$ converges in distribution to $Exp(1)$ (we denote by $Exp(r), r>0$, the exponential variable with parameter $r$)\cite{MR2554368, dhersin2012external}.
\item $\Lambda(dx)/dx=\frac{x^{a-1}(1-x)^{b-1}}{Beta(a,b)}\ind_{0\leq x\leq 1}, 0<a< 1,b>0:$ $Beta(a,b)$-coalescent. Here $Beta(\cdot,\cdot)$ denotes Euler's beta function.  Then $n^{1-a}T^{(n)}$ converges in distribution to a random variable $T(a,b)$ which has density function $\frac{\Gamma(a+b)}{(1-a)\Gamma(b)}(1+\frac{\Gamma(a+b)}{(2-a)\Gamma(b)}x)^{-\frac{3-2a}{1-a}}\ind_{x\geq 0}$ (see \cite{dhersin2012length}). 
\item $\int_0^1x^{-1}\Lambda(dx)<+\infty$: These processes are called coalescents without proper frequencies (\cite{MR1742892}).  This category contains $Beta(a,b)$-coalescents with $a>1, b>0$ (see \cite{MR1742892}, \cite{MR1781024}). Then $\left(\int_0^1x^{-1}\Lambda(dx)\right)T^{(n)}$ converges in distribution to $Exp(1)$ (see \cite{MR2484170}, \cite{MR2684740} ). 
\end{itemize}

We see a common property for the last three cases concerning one external branch length which is that the normalization factor for $T^{(n)}$ is $\mu^{(n)}:=\int_{1/n}^1x^{-1}\Lambda(dx)$. More precisely, 

\begin{itemize}
\item Bolthausen-Sznitman coalescent:  Notice that $\mu^{(n)}=\ln n$. Hence  directly we have $\displaystyle  \mu^{(n)}T^{(n)}\stackrel{(d)}{\rightarrow } Exp(1)$.
\item $Beta(a,b)$-coalescent with $0<a<1, b>0$: 
$$\mu^{(n)}=\int_{1/n}^1x^{-1}\Lambda(dx)=\int_{1/n}^1\frac{\Gamma(a+b)}{\Gamma(a)\Gamma(b)}x^{a-2}(1-x)^{b-1}dx=\frac{\Gamma(a+b)}{(1-a)\Gamma(a)\Gamma(b)}n^{1-a}+O(1).$$ Hence $\mu^{(n)}T^{(n)}$ converges in distribution to $T(a,b)\Gamma(a+b)/(1-a)\Gamma(a)\Gamma(b)$.
\item  If $\int_0^1x^{-1}\Lambda(dx)<+\infty$, then $\displaystyle\lim_{n\rightarrow +\infty} \frac{\mu^{(n)}}{\int_{0}^1x^{-1}\Lambda(dx)}=1$. Hence $\mu^{(n)}T^{(n)}$ converges in distribution to $Exp(1)$.
\end{itemize}

Kingman coalescent can be viewed as the formal limit of $Beta(a,b)$-coalescent with $0<a<1, b>0$ when $a$ tends to $0$, since the measure $\frac{x^{a-1}(1-x)^{b-1}dx}{Beta(a,b)}\ind_{0\leq x\leq 1}$ tends weakly to the Dirac measure on $0$.  The normalization factor in the case of $Beta(a,b)$-coalescent is $n^{1-a}$, and of Kingman coalescent is $n$. Then we see that these two factors show also some kind of continuity as $a$ tends to $0$.  We can formally take $n$ as $\mu^{(n)}$ in the case of Kingman coalescent. 

Therefore $\mu^{(n)}$ is characteristic for the randomly chosen external branch length in those processes considered.  Notice that $\mu^{(n)}$ concerns only the measure $\Lambda\ind_{[1/n,1]}$, so it is natural to think about the influences of measures $\Lambda\ind_{[1/n,1]}$ and $\Lambda\ind_{[0, 1/n)}$ on the external branch lengths.  More generally,  if $\Lambda=\Lambda_1+\Lambda_2$, how can we evaluate each influence on the construction of the whole $\Lambda$-coalescent? If $\Lambda_1$ is ''small'' enough, we can imagine that $\Pi^{(n)}$ looks like $\Pi^{(2, n)}$ (recall that $\Pi^{(2, n)}$ is the $\Lambda_2$-coalescent ).  In this case, we call $\Lambda_1$ the noise measure and $\Lambda_2$ the main measure.  To separate $\Lambda_1$ and $\Lambda_2$, we introduce in the next section the "measure division construction" of a $\Lambda$-coalescent. The idea of this construction can be at least tracked back to \cite{berestycki-2008-44} where {the authors consider also a coupling of two finite measures on $[0,1]$ but in a slightly different manner. } 

The main results are as follows:
\begin{theo}\label{gnmu0}
If $\Lambda$ satisfies: 
\begin{equation}\label{gnmu}\displaystyle \lim_{n\rightarrow +\infty }\frac{g_n}{n\mu^{(n)}}=0,  \end{equation}
 then $\displaystyle \mu^{(n)}T^{(n)}\stackrel{(d)}{\rightarrow }Exp(1)$. 
\end{theo}

\begin{rem}\label{nonzero} 
\begin{itemize}
\item Condition $(\ref{gnmu})$ implies that $\Lambda(\{0\})=0.$ Indeed, if $\Lambda(\{0\})>0$, then $g_n\geq {n\choose 2}\Lambda(\{0\})$ and $\mu^{(n)}\leq n\Lambda((0,1])$. Then $(\ref{gnmu})$ is invalid.  

\item {The class of coalescents satisfying condition $(\ref{gnmu})$ does not contain the Beta$(a,b)$-coalescents with $0<a<1$ and $b>0$.  The following conjecture uses a description similar to condition $(\ref{gnmu})$ to include them: 

\textbf{Conjecture:} Let $c>0$. If 
$$\lim_{n\rightarrow +\infty}\frac{g_n}{n\mu^{(n)}}=c,$$
then $\mu^{(n)}T^{(n)}\stackrel{(d)}{\rightarrow }T_c$, where $T_c$ is a random variable with density $\Gamma(2-\alpha^*)(1+cx)^{-\frac{\alpha^*}{\alpha^*-1}-1}\ind_{x\geq 0}$. Here $\alpha^*$ is the unique solution of the equation 
$$\frac{(\alpha-1)\Gamma(2-\alpha)}{\alpha}=c.$$

This conjecture is true for Beta$(a,b)$-coalescents with $0<a<1, b>0$. In this case, we have $c=\frac{(1-a)\Gamma(a)}{2-a}$. The coalescents,  which are more general than but similar to Beta$(a,b)$-coalescents with $0<a<1, b>0$, studied in \cite{dhersin2012length} also satisfy this conjecture.}
\end{itemize}
\end{rem}

\textbf{Examples:} We give a short list of typical examples satisfying condition $(\ref{gnmu})$ which are processes without proper frequencies or similar to Bolthausen-Szitman coalescent. {Define $\bar{\mu}^{(n)}:=\int_{1/n}^1x^{-2}\Lambda(dx)$. }

\textbf{Ex 1:}   $\int_0^1x^{-1}\Lambda(dx)<+\infty$: It suffices to prove that $\displaystyle \lim_{n\rightarrow +\infty}\frac{g_n}{n}=0.$ Recalling the expression  (\ref{gn1}) of $g_n$, we have, for $n\geq 2,$
\begin{align}\label{gnn}
\frac{g_n}{n}&=\frac{\int_0^1(1-(1-x)^{n}-nx(1-x)^{n-1})x^{-2}\Lambda(dx)}{n}\nonumber\\
&=\frac{\int_{1/n}^1(1-(1-x)^{n}-nx(1-x)^{n-1})x^{-2}\Lambda(dx)}{n}+\frac{\int_0^{1/n}(1-(1-x)^{n}-nx(1-x)^{n-1})x^{-2}\Lambda(dx)}{n}\nonumber\\
&\leq \frac{\bar{\mu}^{(n)}}{n}+\frac{\int_0^{1/n}n^2\Lambda(dx)}{n}.\\\nonumber
\end{align}

The second term $\frac{\int_0^{1/n}n^2\Lambda(dx)}{n}=\int_0^{1/n}n\Lambda(dx)\leq \int_0^{1/n}x^{-1}\Lambda(dx)\rightarrow 0.$ For the first term, let $\epsilon>0$ and $M=1/\epsilon$ ,  then 
\begin{align*}
 \frac{\bar{\mu}^{(n)}}{n}&= \frac{\int_{M/n}^{1}x^{-2}\Lambda(dx)}{n}+ \frac{\int_{1/n}^{M/n}x^{-2}\Lambda(dx)}{n}\\
 &\leq \frac{\int_{M/n}^{1}x^{-1}\Lambda(dx)}{M}+\int_{1/n}^{M/n}x^{-1}\Lambda(dx)\\
 &\leq \epsilon\int_0^1x^{-1}\Lambda(dx)+\int_{1/n}^{M/n}x^{-1}\Lambda(dx).\\
\end{align*}
Notice that $ \epsilon\int_0^1x^{-1}\Lambda(dx)$ can be arbitrarily  small and $\int_{1/n}^{M/n}x^{-1}\Lambda(dx)$ tends to $0$ as $n$ tends to $+\infty$. Then we get that $\frac{\bar{\mu}^{(n)}}{n}$ tends to $0$.  Hence if $\int_0^1x^{-1}\Lambda(dx)<+\infty$, condition $(\ref{gnmu})$ is satisfied.

\textbf{Ex 2:}  Bolthausen-Sznitman coalescent: In this case, it is straightforward to prove that $g_n=n-1$ and $\mu^{(n)}=\ln n$, then $\displaystyle \lim_{n\rightarrow +\infty}\frac{g_n}{n\mu^{(n)}}=\lim_{n\rightarrow +\infty}\frac{n-1}{n\ln n}=0$.

\textbf{Ex 3:} $\Lambda$ has a density function $f_{\Lambda}$ on $[0,r)$ where $0<r<1$ and there exists a positive number $M$ such that $f_{\Lambda}<M$ on $[0,r)$:  This kind of processes can be considered as being \textit{dominated} by the Bolthausen-Sznitman coalescent. 

If $\int_0^1x^{-1}\Lambda(dx)<+\infty$, we turn back to the first example. If $\int_0^1x^{-1}\Lambda(dx)=+\infty$, then we have $g_n\leq 2M(n-1)$ for $n$ large enough, hence $\displaystyle \limsup_{n\rightarrow +\infty}\frac{g_n}{n\mu^{(n)}}\leq \lim_{n \rightarrow +\infty}\frac{2M(n-1)}{n\mu^{(n)}}=0$. It turns out that this kind of coalescent also satisfies condition $(\ref{gnmu})$.

\textbf{Ex 4:} $\Lambda$ has a density function $f_{\Lambda}(x)=p(\ln\frac{1}{x})^q$ on $[0,r)$ where $0<r<1$ and $p,q$ are positive numbers:
Using (\ref{gnn}), we have
\begin{align*}
\frac{g_n}{n\mu^{(n)}}&\leq  \frac{\bar{\mu}^{(n)}}{n\mu^{(n)}}+\frac{\int_0^{1/n}n^2\Lambda(dx)}{n\mu^{(n)}}, \forall n\geq 2.
\end{align*}
For two real sequences $(x_n)_{n\geq 1}, (y_n)_{n\geq 1}$, we write  $x_n\asymp y_n$, if there exist two positive constants $c, C$ such that $cy_n\leq x_n\leq Cy_n$ for $n$ large enough. Then it is not difficult to find out that $\mu^{(n)}\asymp(\ln n)^{q+1}$, $\bar{\mu}^{(n)}\asymp n(\ln n)^q$, $\int_0^{1/n}n^2\Lambda(dx)\asymp n(\ln n)^q$. Hence we get $\displaystyle \frac{g_n}{n\mu^{(n)}}\rightarrow 0$.

\begin{theo}\label{kd}
If $\Lambda$ satisfies condition $(\ref{gnmu})$ and $\int_0^1x^{-1}\Lambda(dx)=+\infty$, then we have:
\begin{equation}\label{mainkd} \mu^{(n)}(T_1^{(n)},T_2^{(n)},\cdots, T_n^{(n)}, 0, 0, \cdots) \stackrel{(d)}{\rightarrow } (e_1,e_2,\cdots ),\end{equation}
where $(e_i)_{i\in\mathbb{N}}$ are independently distributed as $Exp(1)$.
\end{theo}
\begin{rem}
The same result has been proved for Bolthausen-Sznitman coalescent in \cite{dhersin2012external}. The authors have used a moment method. We can apply this theorem to Example 4 and Example 3 when $\int_0^1x^{-1}\Lambda(dx)=+\infty.$ If $\int_0^1x^{-1}\Lambda(dx)<+\infty,$ then (\ref{mainkd}) is not true and there is no more asymptotic independence (see \cite{mohle2010asymptotic}). 
\end{rem}

The following three corollaries have also been proved for Bolthausen-Sznitman coalescent (see \cite{dhersin2012external}, \cite{MR2353033}, \cite{gnedin2012asymptotics}).

\begin{cor}\label{momk}
If $\Lambda$ satisfies condition $(\ref{gnmu})$, then for any $r\in \mathbb{R}^+,$ 
\begin{equation*} \displaystyle \lim_{n\rightarrow +\infty}\mathbb{E}[(\mu^{(n)}T^{(n)})^r]=\mathbb{E}[e_1^r],\end{equation*}
where $e_1$ is distributed as $Exp(1)$. Moreover, if $\int_0^1x^{-1}\Lambda(dx)=+\infty,$ then for any $k\in\mathbb{N}$ and any $ (r_1,r_2,\cdots, r_k)\in\{\mathbb{R}^+\}^k,$ we have:
\begin{equation}\label{multiple}\displaystyle \lim_{n\rightarrow +\infty}\mathbb{E}[\prod_{i=1}^{k}(\mu^{(n)}T_i^{(n)})^{r_i}]=\mathbb{E}[\prod_{i=1}^{k}e_i^{r_i}],\end{equation}where $(e_i)_{1\leq i\leq k}$ are independently distributed as $Exp(1).$
\end{cor}

\begin{cor}\label{ext}
If $\Lambda$ satisfies condition $(\ref{gnmu})$ and $\int_0^1x^{-1}\Lambda(dx)=+\infty$, then the total external branch length $L_{ext}^{(n)}$ satisfies: $\mu^{(n)}L_{ext}^{(n)}/n$ converges in $L^2$ to $1.$
\end{cor}

\begin{cor}\label{total}
If $\Lambda$ satisfies condition $(\ref{gnmu})$ and $\int_0^1x^{-1}\Lambda(dx)=+\infty$, then the total branch length $L_{total}^{(n)}$ satisfies: $\mu^{(n)}L_{total}^{(n)}/n$ converges in probability to $1.$
\end{cor}
\begin{rem}\begin{itemize}
\item
In fact, we will prove that $\displaystyle \lim_{n\rightarrow +\infty}\mathbb{E}[\mu^{(n)}L_{total}^{(n)}/n]=1$. Notice that Corollary \ref{momk} gives $\displaystyle \lim_{n\rightarrow +\infty}\mathbb{E}[\mu^{(n)}L_{ext}^{(n)}/n]=1$. Hence we deduce this corollary using Corollary \ref{ext}. 
\item If $\int_0^1x^{-1}\Lambda(dx)<+\infty$, then (\ref{multiple}) and Corollaries (\ref{ext}) and (\ref{total}) are not true (see again \cite{mohle2010asymptotic}).
\end{itemize}
\end{rem}
\subsection{Organization} In section $2$, we introduce the main object of this paper: the measure division construction. At first,  one needs to define the restriction by the smallest element which serves as a preliminary step of measure division construction. In the same section, we then introduce the two-type $\Lambda$-coalescent which is defined using the measure division construction. This process gives a label \textit{primary} or \textit{secondary} to every block and its every element of a normal $\Lambda$-coalescent.  Using this process, we can see more clearly the coalescent times of some singletons. For a technical use, we then give a tripling to estimate the number of blocks at small times of $\Pi^{(1,n)}$ which is related to the noise measure $\Lambda_1$.  

In section $3$, we at first give a characterization for the condition $(\ref{gnmu})$.  Then we apply the general results obtained in section 2 to those processes satisfying $(\ref{gnmu})$. Finally, we give all the proofs for the results presented in the section $1$.

\section{Measure division construction}
\subsection{Restriction by the smallest element.}
\begin{figure}
  \centering
  \subfigure[$\Pi^{( 5)}$]{\label{fig:coalescent.eps}\includegraphics[width=0.495\textwidth]{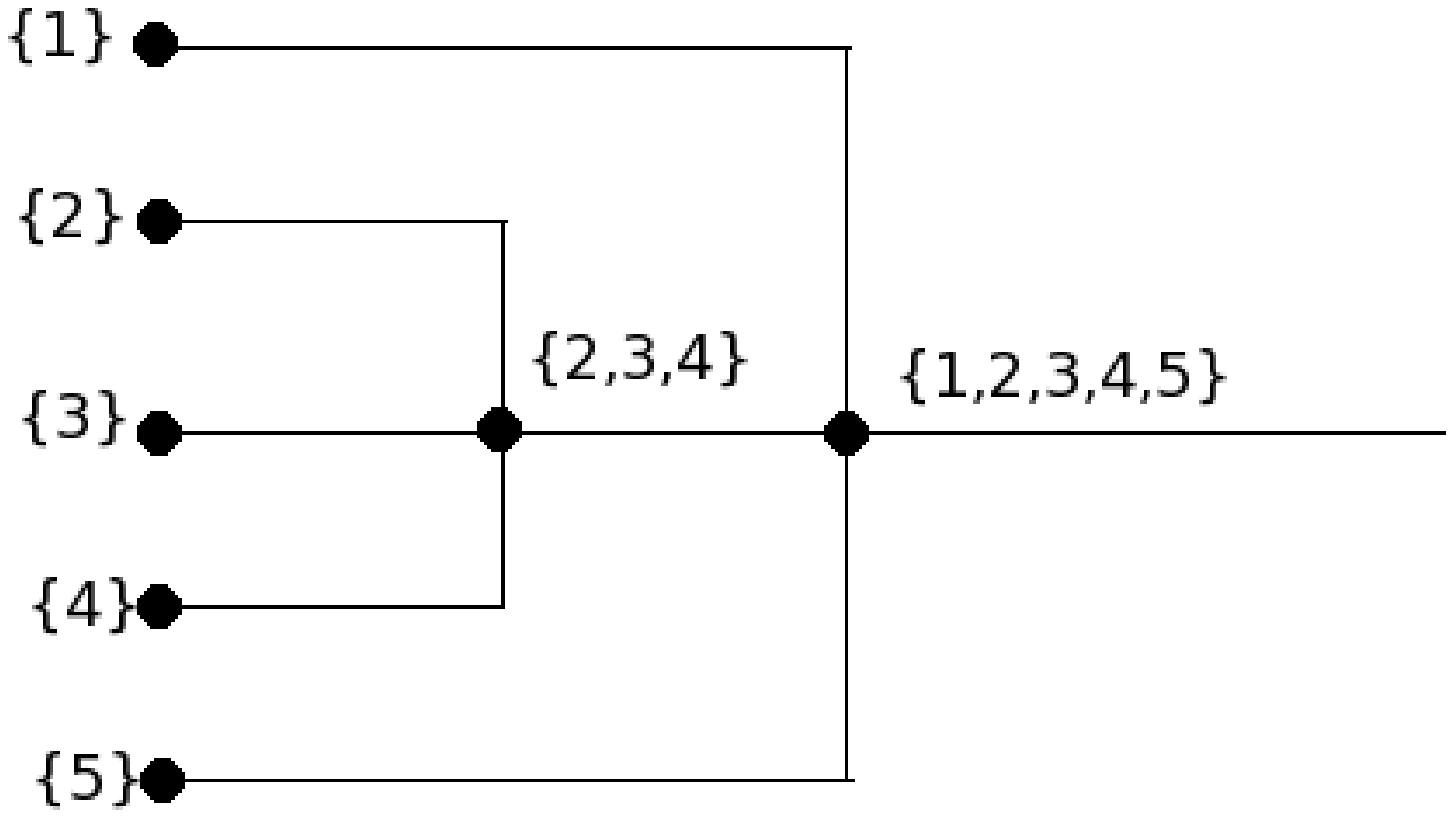}}                
  \subfigure[A restriction by the smallest element of $\Pi^{( 5)}$ from $\{\{1\},\cdots, \{5\}\}$ to $\{\{1,2\},\{3,5\},\{4\}\}$ ]{\label{fig:restrictionsmall.eps}\includegraphics[width=0.495\textwidth]{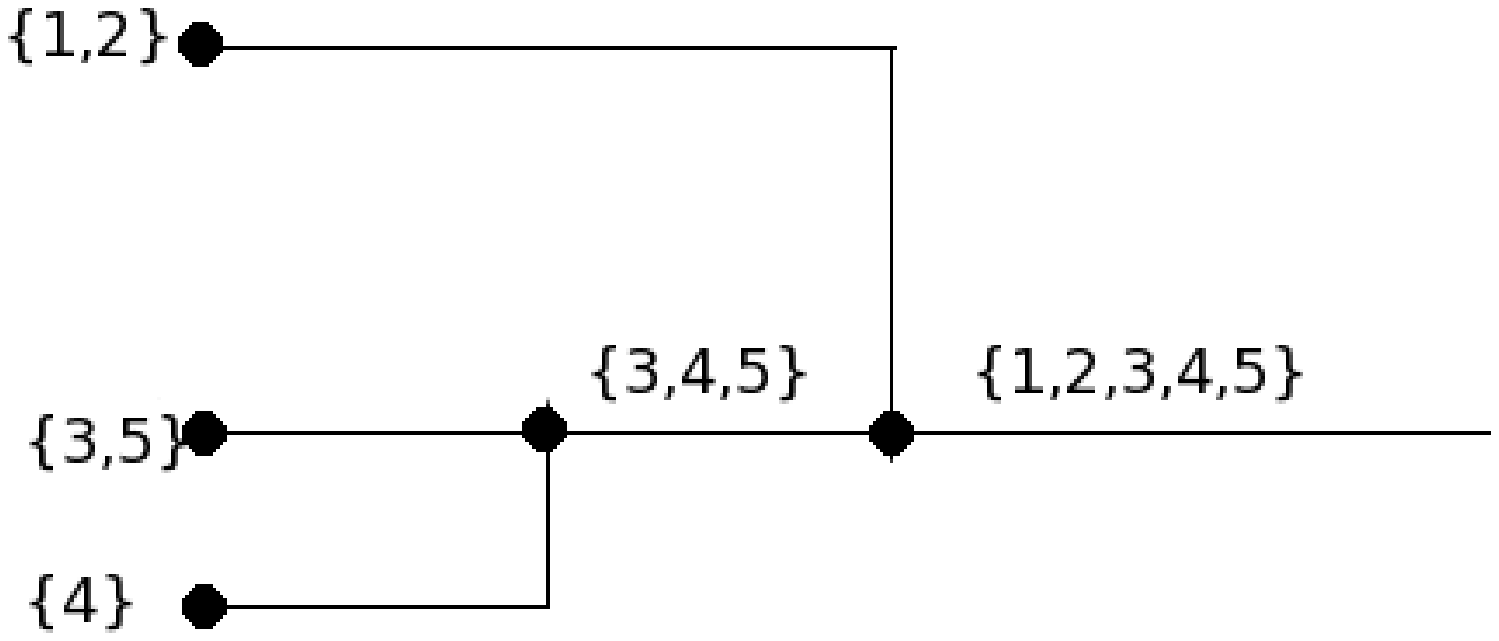}}
  \caption{Restriction by the smallest element}
  \label{fig:Restriction small}
\end{figure}

Let $\xi_n=\{A_1,\cdots, A_{|\xi_n|}\}$, $\chi_n=\{B_1,\cdots, B_{|\chi_n|}\}$ be two partitions of $\{1,2,\cdots, n\}$. We define $s_i^{A}$ (resp. $s_i^{B}$) as the smallest number in the block $A_i$ (resp. $B_i$). We define also the notation $\xi_n\preceq \chi_n,$ if $|\chi_n| \leq  |\xi_n|$ and for any $1\leq i\leq |\chi_n|,  B_{i}=\cup_{j\in I_i}A_j$, where  $\{I_i\}_{1\leq i\leq |\chi_n|}$ is a partition of $\{1,2,\cdots, |\xi_n|\}$. Roughly speaking, $\xi_n$ is finer than $\chi_n$.

If  $\xi_n\preceq \chi_n,$ we define the stochastic process $\bar{\Pi}^{(\chi_n)}$ which is the restriction by the smallest element of $\Pi^{( \xi_n)}$ from $\xi_n$ to $\chi_n$: 
\begin{itemize}
\item $\bar{\Pi}^{(\chi_n)}(0)=\chi_n$;
\item For any $ t\geq 0$, if $\Pi^{( \xi_n)}(t)=\{D_i\}_{1\leq i\leq |\Pi^{( \xi_n)}(t)|}$, where $D_i$ denotes a block,  then
$$\bar{\Pi}^{(\chi_n)}(t)=\{\bigcup_{s_j^{B}\in D_i}B_j\}_{1\leq i\leq |\Pi^{( \xi_n)}(t)|},$$ 
where the empty sets in $\bar{\Pi}^{(\chi_n)}(t)$ are removed.
\end{itemize}

Notice that the restriction by the smallest element is defined from path to path (see Figure \ref{fig:Restriction small}).  

\begin{lem}\label{restby}
$\bar{\Pi}^{(\chi_n)}$ has the same distribution as $\Pi^{( \chi_n)}$. 
\end{lem}
\begin{proof}
Every block in $\chi_n$ is identified by its smallest element which belongs to a unique block in $\xi_n$. Hence for any $B_i$ in $\chi_n$, there exists a unique  $A_{\tau_i}$ such that $A_{\tau_i}\in \xi_n$, $A_{\tau_i} \subset B_i$ and $s_{\tau_i}^{A}=s_i^{B}$ with $\tau_i\in \{1,2,\cdots, |\xi_n|\}.$ Let $\chi'_n=\{A_{\tau_i}\}_{1\leq i\leq |\chi_n|}$ and define a new process $\hat{\Pi}^{(\chi'_n)}$ as follows: 

\begin{itemize}
\item $\hat{\Pi}^{(\chi'_n)}(0)=\chi'_n.$
\item For any $ t\geq 0$, if $\Pi^{( \xi_n)}(t)=\{D_i\}_{1\leq i\leq |\Pi^{( \xi_n)}(t)|}$, then
$$\hat{\Pi}^{(\chi'_n)}(t)=\{\bigcup_{s_{\tau_j}^{A}\in D_i}A_{\tau_j}\}_{1\leq i\leq |\Pi^{( \xi_n)}(t)|},$$ 
where the empty sets in $\hat{\Pi}^{(\chi'_n)}(t)$ are removed.
\end{itemize}

It is easy to see that $\hat{\Pi}^{(\chi'_n)}$ is a { natural} restriction of $\Pi^{( \xi_n)}$ from $\xi_n$ to $\chi'_n$. By the consistency property, we get $\hat{\Pi}^{(\chi'_n)}\stackrel{(d)}{=}\Pi^{( \chi'_n)}$. { In} the construction of $\hat{\Pi}^{(\chi'_n)}$ and $\bar{\Pi}^{(\chi_n)}$, what is determinant is the smallest element in each block.  Hence to obtain $\bar{\Pi}^{(\chi_n)}$ from $\hat{\Pi}^{(\chi'_n)}$, at time $0$, one needs to complete every $A_{\tau_i}$ by some other numbers larger than $s_{\tau_i}^{A}$ to get $B_i$ and then follow the evolution of $\hat{\Pi}^{(\chi'_n)}$. It turns out that $\bar{\Pi}^{(\chi_n)}$ is a coalescent process with initial value $\chi_n$. Hence we can conclude.\end{proof}
 
 \subsection{Measure division construction} 
Let $\Lambda, \Lambda_1,\Lambda_2$ be three finite measures such that $\Lambda=\Lambda_1+\Lambda_2$.  We denote by  $\Pi_{1,2}^{(n)}:=(\Pi_{1,2}^{(n)}(t))_{t\geq 0}$ the stochastic process constructed by the measure division construction using $\Lambda_1$ and $\Lambda_2$. Here  the index $(1,2)$ is for $\Lambda=\Lambda_1+\Lambda_2$ with $\Lambda_1$ called noise measure and $\Lambda_2$ main measure.  Recall that $\Pi^{(1,n)}$ is the $\Lambda_1$-coalescent with $\Pi^{(1,n)}(0)=\{\{1\}, \{2\},\cdots, \{n\}\}$.
\begin{itemize}
\item Step 0: Given a realization or a path $\Pi$ of $\Pi^{(1, n)}$, we set $\Pi_{1,2}^{(n)}(t)=\Pi(t)$, for any $t\geq 0$.  We set also $t_0=0$.
\item Step 1: Let $t_1,t_2, \cdots$ be the coalescent times after $t_0$ of { $\Pi_{1,2}^{(n)}$} (if there is no collision after $t_0$, we set $t_i=+\infty, i\geq 1$).  Within $[t_0,t_1)$, $\Pi_{1,2}^{(n)}$ is constant.  Then we run an independent $\Lambda_2$-coalescent with initial value $\Pi_{1,2}^{(n)}(t_0)$ from time $t_0$.  \begin{itemize}
\item If the $\Lambda_2$-coalescent has no collision on $[t_0,t_1)$, we pass to $[t_1, t_2)$. Similarly,  we construct another independent $\Lambda_2$-coalescent with initial value $\Pi_{1,2}^{(n)}(t_1)$ from time $t_1$, and so on. 
\item Otherwise, we go to the next step.\end{itemize}
\item  Step 2: If finally within $[t_{i-1}, t_{i})$, the related independent $\Lambda_2$-coalescent has its first collision at time $t_{*}$ and its value at $t_*$ is $\xi$. We then modify $(\Pi_{1,2}^{(n)}(t))_{t\geq 0}$ in the following way:
\begin{itemize}
\item We change nothing for $0\leq t<t_{*}$. 
\item Let $\Pi'=(\Pi'(t), t\geq t_{*})$ be the restriction by the smallest element of $(\Pi_{1,2}^{(n)}(t))_{t\geq t_{*}}$ from $\Pi_{1,2}^{(n)}(t_{*})$ to $\xi$. Then let $(\Pi_{1,2}^{(n)}(t))_{t\geq t_{*}}=(\Pi'(t))_{t\geq t_{*}}$ and go to the step 1 by taking $t_*$ as a new starting point. Notice that, due to Lemma \ref{restby},  $(\Pi_{1,2}^{(n)}(t))_{t\geq t_{*}}$ has the same distribution as a $\Lambda_1$-coalescent from time $t_*$ with initial value $\xi$, . 

\end{itemize}
\end{itemize}

\begin{rem}\label{defmdc}
\begin{itemize}
\item The measure division construction works path by path.
\item If we take $\Lambda_1=0$ as noise measure and $\Lambda_2=\Lambda$ as main measure, then $\Pi^{(1, n)}(t)=\{\{1\},\{2\}, \cdots, \{n\}\}$  for any $t\geq 0$ and $\Pi_{1,2}^{(n)}\stackrel{(d)}{=}\Pi^{(n)}$.
\end{itemize}
\end{rem}
\begin{theo}\label{pi12pi}
Let $\Lambda$, $\Lambda_1$ and $\Lambda_2$ be three finite measures and $\Lambda=\Lambda_1+\Lambda_2$. Then we have $\Pi_{1,2}^{(n)}\stackrel{(d)}{=}\Pi^{(n)}$.
\end{theo}
\begin{proof} Let $t$ be a coalescent time of $\Pi_{1,2}^{(n)}$. We consider the time of the next coalescence and the value at that moment.  In the measure division construction of $\Pi_{1,2}^{(n)}$, we can see appearing two independent processes with one being a $\Lambda_1$-coalescent with initial value $\Pi_{1,2}^{(n)}(t)$ and the other one being a $\Lambda_2$-coalescent with initial value $\Pi_{1,2}^{(n)}(t)$ from time $t.$ The process $\Pi_{1,2}^{(n)}$ gets the next coalescence whenever one of them first  encounters a coalescence and picks up the value of the process at that moment. Then we follow the same procedure from the new coalescent time of $\Pi_{1,2}^{(n)}$. It is easy to see that $\Pi_{1,2}^{(n)}$ behaves in the same way as $\Pi^{(n)}$. Hence we can conclude. 
\end{proof}

\begin{rem}
The theorem shows that if we exchange the noise measure and the main measure, the distribution of the process is not changed and is uniquely determined by their sum. 
\end{rem}
\begin{rem}
The measure division construction also works for more than two measures.  If there are $k (k\geq 2)$ finite measures $\{\Lambda_i\}_{1\leq i\leq k}$ and $\Lambda=\sum_{i=1}^{k}\Lambda_i,$ one can get a stochastic process by first giving a realization of $\Pi^{(1, n)}$ which will be modified by $\Lambda_2$ in the way described in the measure division construction, and then we apply $\Lambda_3$ on the modified process, etc. The equivalence in distribution can be obtained in a recursive way. 
\end{rem}

We give a corollary to show an immediate application of the measure division construction. The following corollary is essentially the same as Lemma 3.2 in \cite{berestycki-2008-44}. But we prove it again in our way.
\begin{cor}\label{measto}
Let $\Lambda_1$, $\Lambda_2$ be two finite measures such that $\Lambda_1\leq \Lambda_2$, then on can construct $\Pi^{(1,n)}$ and $\Pi^{(2,n)}$ such that $|\Pi^{(2, n)}(t)|\leq |\Pi^{(1, n)}(t)|$ for all $t\geq 0$.\end{cor}
\begin{proof}
{$\Pi^{(2, n)}$ can be regarded as the measure constructed process by imposing the measure $\Lambda_2-\Lambda_1$ on the paths of $\Pi^{(1, n)}$. Then we can deduce this corollary. }
 \end{proof}

\subsection{Two-type $\Lambda$-coalescents.}
\subsubsection{Definitions}
Let $\Lambda, \Lambda_1$, $\Lambda_2$ be three finite measures and $\Lambda=\Lambda_1+\Lambda_2$ and $\Lambda_2$ satisfies $\int_0^1x^{-2}\Lambda_2(dx)<+\infty$. A two-type $\Lambda$-coalescent, denoted by $\tilde{\Pi}_{1,2}^{(n)},$ is to give a label \textit{primary} or \textit{secondary} to every block and also to its every element at any time $t$ of a normal $\Lambda$-coalescent. A block is \textit{secondary} if and only if every element in this block is \textit{secondary}. The construction is via the measure division construction.  Let $(\eta^{(2)}_i)_{i\geq 1}$ be independent random variables following the distribution of $\frac{x^{-2}\Lambda_2(dx)}{\int_0^1x^{-2}\Lambda_2(dx)}$, $(e_i^{(2)})_{i\geq 1}$ i.i.d copies of $Exp(\int_0^1x^{-2}\Lambda_2(dx))$ and $(S_i^{(2)})_{i\geq 1}=(\sum_{j=1}^{i}e_j^{(2)})_{i\geq 1}$. 

\textbf{Construction of a two-type $\Lambda$-coalescent:}
\begin{itemize}
\item Step 0: We pick a realization or a path $\Pi$ of $\Pi^{(1, n)}$. Every element and every block of $\Pi$ at any time is labeled \textit{primary}.  We also fix independent realizations of $(\eta^{(2)}_i)_{i\geq 1}$ and $(S_i^{(2)})_{i\geq 1}$. Let $\tilde{\Pi}_{1,2}^{(n)}$ be the path $\Pi$ with labels.
\item Step 1: At time $S^{(2)}_1$,  every block of $\tilde{\Pi}_{1,2}^{(n)}(S^{(2)}_1)$ is independently marked "Head" with probability $\eta^{(2)}_1$ and "Tail" with probability $1-\eta^{(2)}_1$. Every element in a "Head" block is then labelled \textit{secondary}. All those blocks marked "Head" are merged into a bigger block, provided that there are at least two "Head"s. In this case, we use the restriction by the smallest element to modify $\tilde{\Pi}_{1,2}^{(n)}$ at time $S_1^{(2)}$in the same way as in the measure division construction in section 2.2. We still call the modified path $\tilde{\Pi}_{1,2}^{(n)}$ and then forward to the time $S^{(2)}_{2}$ and do the same operations. This procedure can be continued until MRCA.

 \end{itemize}
 
 It is easy to verify that without labels, $\tilde{\Pi}_{1,2}^{(n)}$ has the same distribution as $\Pi^{(n)}$.  We call $(S_i^{(2)})_{i\geq 1}$ the \textit{marking time}s. We define $L_i^{(2, n)}$ as the \textit{first marking time} of $\{i\}$ when $\{i\}$ is marked "Head " for the first time.  Let $L_i^{(2, n)}=+\infty$,  if $\{i\}$  is never marked as "Head" . 
 
\begin{rem}
If $\Pi=\{\{1\}, \cdots, \{n\}\}$, then we get a coupling between $\Lambda_2$-coalescent and its related annihilator process (see \cite{MR2353033}).  More precisely,
the whole process without labels is the $\Lambda_2$-coalescent and the restriction to \textit{primary} elements and blocks is the annihilator process. 
\end{rem} 
\subsubsection{Coalescencent times and \textit{first marking times}} The above construction of two-type coalescents shows that coalescences happen only at the \textit{marking times}. This property will help us to understand the coalescent times of singletons in terms of their \textit{first marking times}.

\begin{lem}\label{rela2}
Let $\Pi$ be the path of $\Pi^{(1,n)}$ chosen at the Step 0 of the construction of two-type $\Lambda$ coalescent.  Assume that at some time $t>0$,  $\{1\}\in \Pi(t)$, $|\Pi(t)|=m$ with $2\leq m\leq n$. Let $P_{1,2}^{(n, m)}(t)$ be the probability for $\{1\}$ to be coalesced at its \textit{first marking time} within $[0, t)$. Then we have 
\begin{equation}\label{1}
P_{1,2}^{(n,m)}(t)\geq P_t^{(2, m)}:=\sum_{i=1}^{+\infty}\mathbb{E}[\ind_{S_i^{(2)}<t}\Delta^{(2)}_i\left(1-(1-\Delta^{(2)}_i)^{m-1}\right)],
\end{equation}
where $\Delta_1^{(2)}=\eta_1^{(2)}$; $\Delta_i^{(2)}=\eta_i^{(2)}\prod_{j=1}^{i-1}(1-\eta_j^{(2)})$ for $i>1$. Notice that the parameter $n$ is hidden in $P_t^{(2,m)}$.\end{lem}

\begin{proof}Let $i_1,\cdots, i_m$ be the $m$ smallest elements respectively in each block at time $t$ with $1=i_1\leq i_2\leq \cdots \leq i_m\leq n$.

Conditional on $(S_i^{(2)}, \eta_i^{(2)})_{i\geq 1}$,  $\Delta_i^{(2)}$ is the probability for $\{1\}$ to have its \textit{first marking time} at $S_i^{(2)}$ (assume that $S_i^{(2)}\leq t$). To let $\{1\}$ be coalesced at $S_i^{(2)}$, one needs also at least one other block marked "Head" at that time. To get a lower bound of $P_{1,2}^{(n,m)}(t)$, one can consider the propability to have at least one \textit{primary} block containing one element of $\{i_1,\cdots, i_m\}$ to be marked "Head" at that time and this probability is $1-(1-\Delta^{(2)}_i)^{m-1}.$
\end{proof}

\begin{lem}\label{rela1}
In addition to the assumptions in the previous lemma, we assume further that all $\{i\}\in \Pi(t)$ for $1\leq i\leq k$ and $1\leq k \leq m$. Define the probability $P_{1,2}^{(n, m, k)}(t)$ for every $\{i\}$ to be coalesced at its \textit{first marking time} within $[0,t)$. Then we have

\begin{equation}\label{1k}
P_{1,2}^{(n, m, k)}(t)\geq 1-k(1-P_t^{(2, m)}).
\end{equation} 
\end{lem}
\begin{proof}
Let $E=\{\forall 1\leq i\leq k, \{i\}\in \Pi(t); |\Pi(t)|=m\}$,  which denotes the assumptions of $\Pi(t)$ in this Lemma. Then 
\begin{align*}
P_{1,2}^{(n, m, k)}(t)&=\mathbb{P}(\{1\},\cdots, \{k\} \, \text{coalesce at their}  \,\textit{first marking times}  \, \text{within} \,[0,t)|E)\\
&=1-\mathbb{P}(\text{one of} \, \{\{1\},\cdots, \{k\}\} \, \text{does not coalesce at its} \, \textit{first marking times} \, \text{within} [0,t)|E)\\
&\geq 1-\mathbb{P}(\text{none of} \{\{1\},\cdots, \{k\}\} \, \text{coalesce at their} \, \textit{first marking times}\, \text{within} [0,t)|E)\\
&\geq 1-\sum_{i=1}^k\mathbb{P}( \{i\} \,\text{does not coalesce at its} \, \textit{first marking time}\, \text{within} [0,t)|E)\\
&=1-k(1-\mathbb{P}( \{1\} \,\text{coalesces at its} \, \textit{first marking time}\, \text{within} [0,t)|E))\\
&\geq 1-k(1-P_t^{(2, m)}).
\end{align*}

The  last inequality is due to the fact that 
$$\mathbb{P}( \{1\} \,\text{coalesces at its} \, \textit{first marking time}\, \text{within} [0,t)|E))\geq P_t^{(2, m)},$$
which is true due to the same arguments used in the proof of the last Lemma. 

\end{proof}

If $m, t$ are large enough such that under some assumptions, we could prove that $P_t^{(2,m)}$ is very close to $1$.  Then the coalescent times are almost the \textit{first marking times} which are easier to deal with. In section 3.3, we will see such a situation for $\Lambda$ satisfying condition $(\ref{gnmu})$ and $\Lambda_1=\Lambda\mathbf{1}_{[0,1/n)}, \Lambda_2=\Lambda\mathbf{1}_{[1/n,1]}.$ The following corollary studies the \textit{first marking times} in this particular case.
\begin{cor}\label{lll}
Let $t>0$ and $1\leq k\leq n$. Assume that $\Lambda$ satisfies condition $(\ref{gnmu})$ and $\Lambda_1=\Lambda\ind_{[0,1/n)}$, $\Lambda_2=\Lambda\ind_{[1/n, 1]}$. Let $\Pi$ be a path of $\Pi^{(1, n)}$. Recall that $L_i^{(2, n)}$ is the \textit{first marking time} of $\{i\}$ for $1\leq i\leq n.$
\begin{itemize}
\item If $\{1\}\in \Pi(t/\mu^{(n)})$, then for any $0\leq t_1\leq t$, $\mathbb{P}(L_1^{(2, n)}\geq \frac{t_1}{\mu^{(n)}}|\Pi)=e^{-t}$
\item Assume moreover $\int_0^1x^{-1}\Lambda(dx)=+\infty$ and $\{i\}\in \Pi(t/\mu^{(n)})$ for any $1\leq i\leq k$ with $1\leq k\leq n$ and fixed. Let $0\leq t_1\leq t_2\leq \cdots \leq t_k\leq t$, we then have 
\begin{equation}\label{mis2}
\lim_{n\rightarrow +\infty}\mathbb{P}(L_i^{(2, n)}\geq \frac{t_i}{\mu^{(n)}}, \forall 1\leq i\leq k|\Pi)=e^{-\sum_{i=1}^kt_i}.
\end{equation}
\end{itemize}
\end{cor}

\begin{proof}
The first case is easy to see, due to the definition of $L_1^{(2,n)}$. For the second case, we only consider $k=2.$ For $k>2$, the proof is similar. Assume that within $[0, t_1/\mu^{(n)}]$, there are $N_1$ \textit{marking time}s and for $(t_1/\mu^{(n)}, t_2/\mu^{(n)}]$, there are $N_2$ \textit{marking time}s. $N_1$ and $N_2$ are independently Poisson distributed with parameters respectively $\frac{\bar{\mu}^{(n)}t_1}{\mu^{(n)}}$ and $\frac{\bar{\mu}^{(n)}(t_2-t_1)}{\mu^{(n)}}$ (here we have $\bar{\mu}^{(n)}=\int_{1/n}^1x^{-2}\Lambda(dx)=\int_{1/n}^1x^{-2}\Lambda_2(dx)$).  Then we get 

\begin{align*}
&\mathbb{P}(L_1^{(2, n)}\geq t_1/\mu^{(n)}, L_2^{(2, n)}\geq t_2/\mu^{(n)}|\Pi)\\&=\mathbb{E}[\Pi_{i=1}^{N_1}(1-\eta^{(2)}_i)^2\Pi_{i=N_1+1}^{N_1+N_2}(1-\eta^{(2)}_i)]\\
&=\mathbb{E}[(1-2\mathbb{E}[\eta^{(2)}_1]+\mathbb{E}[(\eta^{(2)}_1)^2])^{N_1}]\mathbb{E}[(1-\mathbb{E}[\eta^{(2)}_1])^{N_2}]\\
&=exp\left(\frac{\bar{\mu}^{(n)}t_1}{\mu^{(n)}}(-2\mathbb{E}[\eta^{(2)}_1]+\mathbb{E}[(\eta^{(2)}_1)^2])\right)exp\left(\frac{\bar{\mu}^{(n)}(t_2-t_1)}{\mu^{(n)}}(-\mathbb{E}[\eta^{(2)}_1])\right),
\end{align*}
where the last equality is due to the probability generating function of Poisson distribution. Recall that $\mathbb{E}[\eta^{(2)}_1]=\frac{\mu^{(n)}}{\bar{\mu}^{(n)}}$ and $\mathbb{E}[(\eta^{(2)}_1)^2]=\frac{\int_{1/n}^1\Lambda(dx)}{\bar{\mu}^{(n)}}$. Therefore,
$$\frac{\bar{\mu}^{(n)}}{\mu^{(n)}}\mathbb{E}[\eta^{(2)}_1]=\frac{\int_{1/n}^{1}x^{-1}\Lambda(dx)}{\mu^{(n)}}=1, \text{and}\quad \frac{\bar{\mu}^{(n)}}{\mu^{(n)}}\mathbb{E}[(\eta^{(2)}_1)^2]=\frac{\int_{1/n}^{1}\Lambda(dx)}{\mu^{(n)}}\stackrel{}{\rightarrow 0}. $$

Then we can conclude (\ref{mis2}). 

\end{proof}

\subsection{A tripling}
We often have some results on the coalescent related to a special measure, for example, the $Beta$-coalescent. When the process is perturbed by a noise measure, we would wonder whether this damage is negligible. One example is to estimate the number of blocks of the coalescent related to the noise measure. To this aim, we use the tool of tripling. 

\textbf{Tripling:} Notice that  $\Pi^{(n)}$ encounters its first collision after time $e_1^{(n)}$, which is a random variable.  At this collision, the number of blocks is reduced to $n-W_1^{(n)}$, where $W_1^{(n)}$ is a positive integer valued random variable. Then we add $W_1^{(n)}$ new blocks (these blocks can contain any number belonging to $\{n+1, n+2, \cdots \}$) and consider the whole new $n$ ones.  By the consistency property, the evolution of the original $n-W_1^{(n)}$ blocks can be embedded into that of the new $n$ blocks, i.e. after time $e_2^{(n)}$, we have the collision in the new $n$ blocks whose total number is reduced to $n-W_2^{(n)}$ and  we can calculate the distribution of the number of blocks coalesced among the original $n-W_1^{(n)}$ blocks (we call any block containing at least one of $\{1,2,\cdots, n\}$ as "original block" and it is very possible that nothing  happens for the $n-W_1^{(n)}$ blocks). Then we add again new blocks containing different elements to have another $n$ ones. This procedure is stopped when every element of $\{1,2,\cdots, n\}$ is contained in one block. By the definition of $\Lambda$-coalescent,  $(e_i^{(n)})_{i\geq 1}$ are independent exponential random variables with parameter $g_n$ and $(W_i^{(n)})_{i\geq 1}$ are  i.i.d copies of $X_1^{(n)}$.  

The above procedure gives a tripling of $(e_i^{(n)})_{i\geq 1}$, $(W_i^{(n)})_{i\geq 1}$ and $\Pi^{(n)}$. We define $V_i^{(n)}:=\sum_{j=1}^{i}e_j^{(n)}, i\in \mathbb{N}.$ Then we have the following proposition:

\begin{prop}\label{tripleprop} Suppose that $(e_i^{(n)})_{i\geq 1}$, $(W_i^{(n)})_{i\geq 1}$ and $\Pi^{(n)}$ are tripled, then at any time $t\geq 0$, we have 
\begin{equation}\label{triple}
n-\sum_{i=0}^{N(\Lambda,n, t)}W_i^{(n)}\leq |\Pi^{(n)}(t)|,
\end{equation}
where $N(\Lambda,n, t):=card\{i| V_i^{(n)}\leq t \}$, which is Poisson distributed with parameter $g_nt$ and independent of $(W_i^{(n)})_{i\geq 1}$. Meanwhile, 
\begin{equation}\label{w}\mathbb{E}[W_i^{(n)}]=\frac{n\int_0^1(1-(1-x)^{n-1})x^{-1}\Lambda(dx)}{g_n}-1, \text{and} \quad\mathbb{E}[(W_i^{(n)})^2]=\frac{n(n-1)\int_0^1\Lambda(dx)}{g_n}-\mathbb{E}[W_i^{(n)}].\end{equation}
\end{prop}
\begin{proof}
The number of $i$s within $[0,t]$ follows the Poisson distribution with parameter $g_nt$. Due to the tripling, at any time $V_i^{(n)}$ with $0\leq V_i^{(n)}\leq t$, the decrease of number of  blocks (i.e. $|\Pi^{(n)}(V_i^{(n)}-)|-|\Pi^{(n)}(V_i^{(n)})|$) of original blocks is less than or equal to $W_i^{(n)}$. Hence we get (\ref{triple}). Notice that $W_i^{(n)}\stackrel{(d)}{=} X_1^{(n)}$,  then (\ref{w}) is a consequence of two equalities in \cite{DDS2008} with Eq (17) for the first one and p.1007 for the second one. \end{proof}

\section{Applications to coalescents satisfying condition $(\ref{gnmu})$}
\subsection{Characterization of condition $(\ref{gnmu})$.} Some notations for this section: Let $\Lambda$ be a finite measure on $[0,1]$ and $\Lambda_1=\Lambda\ind_{[0,1/n)}$, $\Lambda_2=\Lambda\ind_{[1/n, 1]}$ ; $\mu^{(1/y)}=\int_y^1x^{-1}\Lambda(dx)$, $g_{1/y}=\int_{0}^{1}(1-(1-x)^{1/y}-\frac{1}{y}x(1-x)^{1/y-1})x^{-2}\Lambda(dx)$ with $0<y\leq 1.$ Notice that the definitions of $\mu^{(1/y)}$ and $g_{1/y}$ are consistent with that of $\mu^{(n)}$ and $g_n$ when $\Lambda(\{0\})=0$. These notations help to examine carefully different measures. 

Here we are going to prove Theorem \ref{gnmu0}, Theorem \ref{kd}, Corollary \ref{momk}, Corollary \ref{ext} and Corollary \ref{total}. Under condition $(\ref{gnmu})$, we decompose $\Lambda$ into $\Lambda_2$ and $\Lambda_1$. The idea is to construct $\Pi^{(n)}$ using measure division construction with noise measure $\Lambda_1$ and main measure $\Lambda_2$. At first, we need to show more details implied by condition $(\ref{gnmu})$. For any real number $x$, let $\lfloor x \rfloor=\max\{y; y\in \mathbb{Z}, y\leq x\}$ and $\lceil x \rceil=\min\{y; y\in \mathbb{Z}, y\geq x\}$
\begin{prop}\label{muequi} The following two assertions are equivalent:

$(*)$: $\Lambda$ satisfies condition $(\ref{gnmu})$;

$(**)$: $\Lambda(\{0\})=0$ and there exists a c\`agl\`ad (limit from right, continuous from left) function $f:[0,1]\rightarrow [0,1]$, continuous at $0$ with $f(0)=0$ such that $\int_0^1\mu^{(1/x)}dx<+\infty$ and

\begin{equation}\label{desmu}\mu^{(1/y)}=\left(\int_0^1\mu^{(1/x)}dx\right)exp\left(\int_y^{1}\frac{f(t)}{t}dt)(1-f(y)\right), 0<y\leq 1. \end{equation}
 \end{prop}

\begin{proof}

\textbf{Part 1}:We first assume that $(*)$ is true. If $\Lambda$ satisfies $(\ref{gnmu})$, then $\Lambda(\{0\})=0$ due to Remark \ref{nonzero}. For $\mu^{(n)}\neq 0$, we have

\begin{align*}
\frac{g_n}{n\mu^{(n)}}&=\frac{\int_0^1(1-(1-x)^n-nx(1-x)^{n-1})x^{-2}\Lambda(dx)}{n\mu^{(n)}}=I_1^{(n)}+I_2^{(n)},
\end{align*}
where $I_1^{(n)}=\frac{\int_{1/n}^1(1-(1-x)^n-nx(1-x)^{n-1})x^{-2}\Lambda(dx)}{n\mu^{(n)}}$, $I_2^{(n)}=\frac{\int_0^{1/n}(1-(1-x)^n-nx(1-x)^{n-1})x^{-2}\Lambda(dx)}{n\mu^{(n)}}$. Notice that for $n$ large,  using monotone property, we have $\frac{e-2}{2e}\frac{\int_{1/n}^1x^{-2}\Lambda(dx)}{n\mu^{(n)}}\leq I_1^{(n)}\leq \frac{\int_{1/n}^1x^{-2}\Lambda(dx)}{n\mu^{(n)}}$ and $\frac{1}{3}\frac{n\int_0^{1/n}\Lambda(dx)}{\mu^{(n)}}\leq I_2^{(n)}\leq \frac{n\int_0^{1/n}\Lambda(dx)}{\mu^{(n)}}.$ Hence condition $(\ref{gnmu})$ is equivalent to 
\begin{equation}\label{2con}
\lim_{n\rightarrow +\infty}\frac{\int_{1/n}^1x^{-2}\Lambda(dx)}{n\mu^{(n)}}=0, \text{and} \lim_{n\rightarrow +\infty}\frac{n\int_0^{1/n}\Lambda(dx)}{\mu^{(n)}}=0, \Lambda(\{0\})=0.
\end{equation}
Then we deduce that 
\begin{equation}\label{secon}
\lim_{y\rightarrow 0+}\frac{\int_0^y\Lambda(dx)}{y\mu^{(1/y)}}=0, \Lambda(\{0\})=0.
\end{equation}
Indeed, for $ 1/y > 2$ and $\mu^{(\lfloor 1/y \rfloor)}\neq 0$, we have
$$\frac{\int_0^y\Lambda(dx)}{y\mu^{(1/y)}}= \frac{\int_0^y\Lambda(dx)}{y\int_{y}^{1}x^{-1}\Lambda(dx)}\leq \frac{\int_0^{1/\lfloor 1/y \rfloor}\Lambda(dx)}{\frac{1}{\lceil 1/y \rceil}\int_{1/\lfloor 1/y \rfloor}^1x^{-1}\Lambda(dx)}=\frac{\lceil 1/y \rceil}{\lfloor 1/y \rfloor}\frac{\lfloor 1/y \rfloor\int_0^{1/\lfloor 1/y \rfloor}\Lambda(dx)}{\int_{1/\lfloor 1/y \rfloor}^1x^{-1}\Lambda(dx)}\stackrel{y\rightarrow 0+}{\rightarrow}0. $$
One thing to notice is that $\displaystyle \lim_{y\rightarrow 0+}y\mu^{(1/y)}=0$ is true for any finite $\Lambda.$ In fact, for any positive number $M$ and $yM<1$, we have 
\begin{align*}
y\mu^{(1/y)}=y\int_y^1x^{-1}\Lambda(dx)=y\int_{yM}^1x^{-1}\Lambda(dx)+y\int_y^{yM}x^{-1}\Lambda(dx)\leq \frac{\int_0^1\Lambda(dx)}{M}+\int_y^{yM}\Lambda(dx),
\end{align*}
where both terms can be made as small as we want by taking $M$ large enough and $y$ close enough to $0$.
Looking into details of $\frac{\int_0^y\Lambda(dx)}{y\mu^{(1/y)}}$ when $\mu^{(1/y)}\neq 0$, we have the following equality, using integration by parts and $\displaystyle \lim_{y\rightarrow 0+}y\mu^{(1/y)}=0$,
\begin{align}\label{parts}
\frac{\int_0^y\Lambda(dx)}{y\mu^{(1/y)}}&=\frac{\int_0^yxx^{-1}\Lambda(dx)}{y\mu^{(1/y)}}=\frac{\int_0^y\mu^{(1/x)}dx-y\mu^{(1/y+)}}{y\mu^{(1/y)}},
\end{align}
where $\mu^{(1/y+)}=\mu^{(1/y)}-y^{-1}\Lambda(\{y\})$.  Due to  (\ref{secon}), we get that $1\geq \frac{\mu^{(1/y+)}}{\mu^{(1/y)}}=1-\frac{\Lambda(\{y\})}{y\mu^{(1/y)}}\geq 1-\frac{\int_0^y\Lambda(dx)}{y\mu^{(1/y)}}\rightarrow 1$. Therefore,  (\ref{secon}) and (\ref{parts}) give
$$\displaystyle \lim_{y\rightarrow 0+}\frac{y\mu^{(1/y)}}{\int_0^y\mu^{(1/x)}dx}=1. $$
Notice that $\int_0^y\mu^{(1/x)}dx\geq y\mu^{(1/y)}$ and $\mu^{(1/y)}$ is a c\`agl\`ad function. Hence there exists a c\`agl\`ad function $f:[0,1]\rightarrow [0,1]$, continuous at $0$ with $f(0)=0$ such that 
\begin{equation}\label{befinal}
\frac{y\mu^{(1/y)}}{\int_0^y\mu^{(1/x)}dx}=1-f(y).
\end{equation}

Now let $G(t)=\int_0^t\mu^{(1/x)}dx$ and any derivative will be considered as left derivative. 
Then (\ref{befinal}) becomes

$$(\ln G(t))'=\frac{G(t)'}{G(t)}=\frac{1-f(t)}{t}.$$ 
Using the fundamental theorem of Newton and Leibniz which also works for c\`agl\`ad functions whose primitive functions take left derivatives. Then for $0<y\leq 1$,
$$\ln G(1)-\ln G(y)=\int_{y}^1(\ln G(t))'dt=\int_y^1\frac{1-f(t)}{t}dt.$$ 
Therefore, 
$$G(y)=G(1)exp(-\int_y^1\frac{1-f(t)}{t}dt).$$
By taking the left derivatives on the both sides and noticing that $G(1)=\int_0^1\mu^{(1/x)}dx$, we can conclude.

\textbf{Part 2}: We now assume that $(**)$ is true. In the first part, we proved implicitly that (\ref{secon}) is equivalent to the $(**)$. Hence we will use (\ref{secon}) to prove (\ref{2con}) which is equivalent to condition $(\ref{gnmu})$ and only the first convergence in (\ref{2con}) is needed to be proved. Let $M$ be a positive number and $\frac{M}{n}\leq 1$, $\mu^{(n)}\neq 0$, then 
\begin{align*}
\frac{\int_{1/n}^1x^{-2}\Lambda(dx)}{n\mu^{(n)}}&=\frac{\int_{M/n}^{1}x^{-2}\Lambda(dx)}{n\mu^{(n)}}+\frac{\int_{1/n}^{M/n}x^{-2}\Lambda(dx)}{n\mu^{(n)}}\\
&\leq \frac{1}{M}+1-\frac{\mu^{(n/M)}}{\mu^{(n)}}.
\end{align*}
The first term can be made as small as we want by taking $M$ large, and the third term $\frac{\mu^{(n/M)}}{\mu^{(n)}}=exp(-\int_{1/n}^{M/n}\frac{f(x)}{x}ds)\frac{1-f(M/n)}{1-f(1/n)}.$ Let $\epsilon>0$ and $n$ large enough such that $f(x)\leq \epsilon$ on $[0,  M/n]$. Then $\frac{\mu^{(n/M)}}{\mu^{(n)}}\geq exp(-\epsilon\ln M)(1-\epsilon),$ which can be made as close as possible to $1$ with $\epsilon$ small enough. Hence we can conclude.
\end{proof}
The next corollary is immediate. 
\begin{cor}\label{maincor}
If $\Lambda$ satisfies $(\ref{gnmu})$, then 
\begin{itemize}
\item  $\displaystyle \lim_{n\rightarrow +\infty}\frac{(\mu^{(n)})^k}{n}=0, \forall k>0$;\\
\item $\displaystyle  \lim_{n\rightarrow +\infty}\frac{\mu^{(n)}}{\mu^{(n-M)}}=1, \forall M>0$; \\
\item $\displaystyle  \lim_{n\rightarrow +\infty}\frac{\mu^{(n)}}{\mu^{(n\epsilon)}}=1, \forall 0<\epsilon<1$.
\end{itemize}

\end{cor}

\subsection{Properties of $\Pi^{(1, n)}$.}

We should next estimate the coalescent process related to the noise measure $\Lambda_1$ which serves as a perturbation to the main measure $\Lambda_2$. At first, one needs a technical result.
\begin{lem}\label{lemgnlambda}We assume that $\Lambda(\{0\})=0.$ Let $g_n^{(1)}=\int_0^1(1-(1-x)^n-nx(1-x)^{n-1})x^{-2}\Lambda_{1}(dx)$ in the spirit of (\ref{gn1}). Then there exists a positive constant $C_1$ such that for $n$ large enough
\begin{equation}\label{gnlambda}g_n^{(1)}\geq C_1n^2\int_0^{1/n}\Lambda_1(dx).
\end{equation}
\end{lem}
\begin{proof}
Let $M>2$. We write 
\begin{align*}
g_n^{(1)}&=\int_0^1(1-(1-x)^n-nx(1-x)^{n-1})x^{-2}\Lambda_1(dx)\\
&=\int_0^{\frac{1}{n}}(1-(1-x)^n-nx(1-x)^{n-1})x^{-2}\Lambda_1(dx)\\
&=I_1+I_2,
\end{align*}
where $I_1=\int_0^{\frac{1}{nM}}(1-(1-x)^n-nx(1-x)^{n-1})x^{-2}\Lambda_1(dx)$ and $I_2=\int_{\frac{1}{nM}}^{\frac{1}{n}}(1-(1-x)^n-nx(1-x)^{n-1})x^{-2}\Lambda_1(dx).$ It is easy to see that for $n\geq 2$, 
\begin{align*}
I_1&\geq \int_0^{\frac{1}{nM}}(n(n-1)-n(n-1)(n-2)x)\frac{1}{2}\Lambda_1(dx)\\
&\geq \int_0^{\frac{1}{nM}}(n(n-1)-(n-1)(n-2)/M)\frac{1}{2}\Lambda_1(dx)\\
&\geq \frac{1}{4}\int_0^{\frac{1}{nM}}n^2\Lambda_1(dx).
\end{align*}
For the second term, 
\begin{align*}
I_2&\geq \int_{\frac{1}{nM}}^{\frac{1}{n}}(1-(1-\frac{1}{nM})^n-\frac{(1-\frac{1}{nM})^{n-1}}{M})n^2\Lambda_1(dx).
\end{align*}
Notice that for $n$ large, there exists a positive constant $C(M)$ such that 
$$1-(1-\frac{1}{nM})^n-\frac{(1-\frac{1}{nM})^{n-1}}{M}\geq C(M)>0.$$
Hence $I_2\geq C(M)\int_{\frac{1}{nM}}^{\frac{1}{n}}n^2\Lambda_1(dx).$ It suffices to take $C_1=\min\{\frac{1}{4}, C(M)\}$ to conclude.
\end{proof}

The following lemma estimates the coalescent process related to the noise measure $\Lambda_1$ when $\Lambda$ satisfies $(\ref{gnmu})$. Recall that $\Pi^{(1,n)}$ is the $\Lambda_1$-coalescent process with $\Pi^{(1,n)}(0)=\{\{1\},\{2\},\cdots, \{n\}\}$.
\begin{lem}\label{noise}
Assume that $\Lambda$ satisfy $(\ref{gnmu})$. Then for any $M>0$, $0< \epsilon \leq 1$ and $n$ large enough,  we have
\begin{equation}\label{lambda1<}
\mathbb{P}\left(|\Pi^{(1, n)}(M/\mu^{(n)})|\leq n-n\epsilon\right)=o(n^{-1}).
\end{equation}
\end{lem}
\begin{proof} If $\int_0^{1/n_0}\Lambda(dx)=0$ with some $n_0>1$, then for any $n>n_0$, $\Lambda_1$ is the null measure and hence $|\Pi^{(1, n)}(t)|=n$ for any $t\geq 0,$ which proves this lemma.  In consequence, one needs only to consider the case where $\int_0^{1/n}\Lambda(dx)\neq0$ for any $n\geq1.$

We recall $g_n^{(1)}$ defined in Lemma \ref{lemgnlambda}. Let $X_1^{(1,n)}$ be the decrease of the number of blocks at the first coalescence of $\Pi^{(1,n)}$. Thanks to Proposition \ref{tripleprop} where we pick up the notations, 
$$n-\sum_{i=1}^{N{(\Lambda_1, n, M/\mu^{(n)})}}W_i^{(n)}\leq|\Pi^{(1, n)}(M/\mu^{(n)})|,$$
where $N{(\Lambda_1, n, M/\mu^{(n)})}$ is Poisson distributed with parameter $\frac{Mg_n^{(1)}}{\mu^{(n)}}$ independent of $(W_i^{(n)})_{i\geq 1}$ which are i.i.d copies of $X_1^{(1, n)}$. Then we have, for $n$ large,
\begin{align}\label{lambda1n}\mathbb{P}(|\Pi^{(1, n)}(M/\mu^{(n)})|\leq n-n\epsilon)&\leq \mathbb{P}\left(n-\sum_{i=1}^{N{(\Lambda_1, n, M/\mu^{(n)})}}W_i^{(n)}\leq n-n\epsilon\right)\nonumber\\
&=\mathbb{P}\left(\sum_{i=1}^{N{(\Lambda_1, n, M/\mu^{(n)})}}W_i^{(n)}-\frac{g_n^{(1)}M}{\mu^{(n)}}\mathbb{E}[W_1^{(n)}]\geq n\epsilon-\frac{g_n^{(1)}M}{\mu^{(n)}}\mathbb{E}[W_1^{(n)}]\right)\nonumber\\
&\leq \frac{\Var(\sum_{i=1}^{N{(\Lambda_{1}, n, M/\mu^{(n)})}}W_i^{(n)})}{(n\epsilon-\frac{g_n^{(1)}M}{\mu^{(n)}}\mathbb{E}[W_1^{(n)}])^2}=\frac{\frac{Mg_n^{(1)}}{\mu^{(n)}}\mathbb{E}[(W_1^{(n)})^2]}{(n\epsilon-\frac{g_n^{(1)}M}{\mu^{(n)}}\mathbb{E}[W_1^{(n)}])^2},
\end{align}
where the second inequality needs $n\epsilon-\frac{g_n^{(1)}M}{\mu^{(n)}}\mathbb{E}[W_1^{(n)}]>0$ which is justified by the following calculations:
Notice that due to Proposition \ref{tripleprop} and Lemma \ref{lemgnlambda}, for $n$ large enough, 
\begin{equation}\label{w1}\mathbb{E}[W_1^{(n)}]+1\leq \frac{n(n-1)\int_0^{1/n}\Lambda_{1}(dx)}{g_n^{(1)}}\leq \frac{1}{C_1}; \mathbb{E}[(W_1^{(n)})^2]\leq \frac{n(n-1)\int_0^{1/n}\Lambda_{1}(dx)}{g_n^{(1)}}\leq\frac{1}{C_1},\end{equation}
where $C_1$ is the positive constant in Lemma \ref{lemgnlambda}. 

Notice that $(\ref{gnmu})$ gives $\frac{g_n^{(1)}}{n\mu^{(n)}}\leq \frac{g_n}{n\mu^{(n)}}\stackrel{}{\rightarrow }0$. Then together with (\ref{w1}), we have 

$$\frac{g_n^{(1)}M}{\mu^{(n)}}\mathbb{E}[W_1^{(n)}]=o(n), \frac{g_n^{(1)}M}{\mu^{(n)}}\mathbb{E}[(W_1^{(n)})^2]=o(n).$$
Hence $n\epsilon-\frac{g_n^{(1)}M}{\mu^{(n)}}\mathbb{E}[W_1^{(n)}]\asymp n\epsilon$. So the inequality (\ref{lambda1n}) is justified and one deduces that 
$$\mathbb{P}(|\Pi^{(\Lambda_{1}, n)}(M/\mu^{(n)})|\leq n-n\epsilon)=o(n^{-1}).$$

Then we conclude (\ref{lambda1<}). \end{proof}

\subsection{Asymptotics of $P^{(2, m)}_t, P_{1,2}^{(n,m)}(t), P_{1,2}^{(n,m, k)}(t), 2\leq m\leq n, t\geq 0$.} 

These terms are probabilities defined in section 2.3.1, which measure the possibility to make one or several singletons coalesced in their \textit{first marking times} within $[0,t)$. In fact, we will study $P^{(2, m)}_{t/\mu^{(n)}}, P_{1,2}^{(n,m)}(t/\mu^{(n)}), P_{1,2}^{(n,m, k)}(t/\mu^{(n)})$, since we want to prove that the normalization factor of the external branch length is $\mu^{(n)}$. We denote by "$\ll$" the stochastic domination between two real random variables. The following corollary together with the remark at the end play an important role in getting the asymptotics of the three probabilities.
\begin{prop}\label{plambda}Suppose that $\Lambda$ satisfies $(\ref{gnmu})$ and $\displaystyle P^{(2,n)}:=\lim_{t\rightarrow +\infty}P_t^{(2, n)}=\sum_{i=1}^{+\infty}\mathbb{E}[\Delta^{(2)}_i\left(1-(1-\Delta^{(2)}_i)^{n-1}\right)]$. Then

\begin{equation}\label{taun}\lim_{n\rightarrow +\infty}P^{(2,n)}=1. \end{equation}
\end{prop} 
\begin{proof} Recall $(\eta^{(2)}_i)_{i\geq 1}$, $(e_i^{(2)})_{i\geq 1}$, $\{\Delta_i^{(2)}\}_{i\geq 1}$ which are associated to $\Lambda_2$ and defined in section 2.3. At first,  we remark that $\sum_{i=1}^{+\infty}\mathbb{E}[\Delta^{(2)}_i]=1.$ One only needs to prove that $\displaystyle \lim_{n\rightarrow +\infty}\sum_{i=1}^{+\infty}\mathbb{E}[\Delta^{(2)}_i(1-\Delta^{(2)}_i)^{n-1}]=0$.
It is easy to see that $\mathbb{E}[\Delta^{(2)}_i(1-\Delta^{(2)}_i)^{n-1}]=\mathbb{E}[\bar{\Delta}^{(2)}_i(1-\bar{\Delta}^{(2)}_i)^{n-1}]$, where $\bar{\Delta}^{(2)}_i=\eta^{(2)}_1\Pi_{j=2}^{i}(1-\eta^{(2)}_j)$. It is obvious that  $(\bar{\Delta}^{(2)}_i)_{i\geq 1}$ is a Markov chain.  For $s>0$, we define a stopping time
\begin{align*}
\tau_s&=\min\{i|\bar{\Delta}^{(2)}_i\leq 1/s\}\\
&=\min\{i|-\sum_{j=2}^{i}\ln(1-\eta^{(2)}_j)\geq \ln s\eta_1^{(2)}\}\\
&=\min\{i+1|-\sum_{j=1}^{i}\ln(1-\eta^{(2)}_{j+1})\geq \ln s\eta_1^{(2)}\}.\\
\end{align*}

Then we get 
\begin{align}\label{2delta}
\sum_{i=1}^{+\infty}\mathbb{E}[\Delta^{(2)}_i(1-\Delta^{(2)}_i)^{n-1}]&=\mathbb{E}[\sum_{i=1}^{+\infty}\bar{\Delta}^{(2)}_i(1-\bar{\Delta}^{(2)}_i)^{n-1}]\nonumber\\
&=\mathbb{E}[\sum_{i=1}^{\tau_n-1}\bar{\Delta}^{(2)}_i(1-\bar{\Delta}^{(2)}_i)^{n-1}+\sum_{i=\tau_n}^{+\infty}\bar{\Delta}^{(2)}_i(1-\bar{\Delta}^{(2)}_i)^{n-1}].\\\nonumber
\end{align}

Notice that $x(1-x)^{n-1}\leq \frac{1}{n}$ , if $\frac{1}{n}\leq x\leq 1$ and  $x(1-x)^{n-1}\leq x$, if $0\leq x\leq \frac{1}{n}$. Then (\ref{2delta}) gives
\begin{equation}\label{1taun}
\sum_{i=1}^{+\infty}\mathbb{E}[\Delta^{(2)}_i(1-\Delta^{(2)}_i)^{n-1}]\leq \mathbb{E}[\frac{\tau_n-1}{n}+\sum_{i=\tau_n}^{+\infty}\bar{\Delta}^{(2)}_i]\leq \mathbb{E}[\frac{\tau_n-1}{n}]+\frac{1}{\mathbb{E}[n\eta^{(2)}_1]}.
\end{equation}
To calculate $\mathbb{E}[\tau_n]$, we use renewal theory.  Let $\mu=\mathbb{E}[-\ln(1-\eta^{(2)}_1)]$. Depending on whether $\mu$ is finite or not, we separate the discussion into two parts. 
 
\textbf{Part 1}: Assume that $\mu<+\infty.$ We denote by $F(t)$ the distribution function and $f(t)$ the density function of $-\ln(1-\eta^{(2)}_1)$ and $X$ an independent random variable with density function $\frac{1}{\mu}(1-F(t))\ind_{t\geq 0}.$  Let $\epsilon>0$, then using integration by parts,
\begin{equation}\label{0x1}\mathbb{P}(0\leq X\leq \epsilon)=\int_0^{\epsilon}\frac{1-F(t)}{\mu}dt=\frac{\epsilon(1-F(\epsilon))}{\mu}+\frac{\int_0^{\epsilon}tf(t)dt}{\mu}\geq \frac{\int_0^{\epsilon}tf(t)dt}{\mu}.\end{equation}

One can write $\int_0^{\epsilon}tf(t)dt$ in another way

$$\int_0^{\epsilon}tf(t)dt=\frac{\int_{1/n}^{1-e^{-\epsilon}}-\ln(1-x)x^{-2}\Lambda(dx)}{\int_{1/n}^{1}x^{-2}\Lambda(dx)}.$$

Notice that $\mu=\frac{\int_{1/n}^{1}-\ln(1-x)x^{-2}\Lambda(dx)}{\int_{1/n}^{1}x^{-2}\Lambda(dx)}<+\infty,$ then there must exist a large number $\epsilon_0>0$ such that for any $\epsilon\geq\epsilon_0$, 

$$\int_0^{\epsilon}tf(t)dt\geq \frac{1}{2}\frac{\int_{1/n}^{1}-\ln(1-x)x^{-2}\Lambda(dx)}{\int_{1/n}^{1}x^{-2}\Lambda(dx)}=\frac{\mu}{2}.$$

Now together with (\ref{0x1}), one gets
\begin{equation}\label{1/2}\mathbb{P}(0\leq X\leq \epsilon)\geq 1/2, \quad \forall \epsilon\geq \epsilon_0.\end{equation}
We fix $\epsilon\geq \epsilon_0$ and define a new Markov chain $(X-\sum_{j=2}^{i}\ln(1-\eta^{(2)}_j))_{i\geq 1}$ and a stopping time $\tau'_s=\min\{i| X-\sum_{j=1}^{i}\ln(1-\eta^{(2)}_{j+1})\geq \ln s)\}$ for $s>0$. It is clear from the definitions of $\tau_s$ and $\tau_s'$ that
 $$\mathbb{E}[\tau'_{s\eta_1^{(2)}}|X=\epsilon]=\mathbb{E}[\tau_{se^{{-\epsilon}}}-1].$$
 
Then
\begin{align*}
\mathbb{E}[\tau'_{n\eta^{(2)}_1}]&=\mathbb{E}[\tau'_{n\eta^{(2)}_1}\ind_{0\leq X\leq \epsilon}]+\mathbb{E}[\tau'_{n\eta^{(2)}_1}\ind_{X> \epsilon}]\\
&\geq \mathbb{P}(0\leq X\leq \epsilon)\mathbb{E}[\tau_{nexp(-\epsilon)}-1]+\mathbb{E}[\tau'_{n\eta^{(2)}_1}\ind_{X> \epsilon}],\\
\end{align*}

which implies that
\begin{equation}\label{tautau}\mathbb{E}[\tau_{nexp(-\epsilon)}]\leq \frac{\mathbb{E}[\tau'_{n\eta^{(2)}_1}]}{\mathbb{P}(0\leq X\leq \epsilon)}+1.\end{equation}

Due to  (4.4) and (4.6) in [\cite{MR0270403}, p.369],  we have 
 $$\mathbb{E}[\tau'_s]=\frac{\ln s}{\mu}, \forall s\geq 1.$$

Notice that $\eta_1^{(2)}\geq \frac{1}{n}$, hence $n\eta_1^{(2)}\geq1.$ Therefore, (\ref{tautau}) gives
\begin{equation}\label{taunmain}
\mathbb{E}[\tau_{n}]\leq \frac{\mathbb{E}[\tau'_{ne^{\epsilon}\eta^{(2)}_1}]}{\mathbb{P}(0\leq X\leq \epsilon)}+1=\frac{\mathbb{E}[\ln(ne^{\epsilon}\eta^{(2)})]}{\mu\mathbb{P}(0\leq X\leq \epsilon)}+1.\end{equation}

Notice that for any $0\leq x< 1$, we have $-\ln(1-x)\geq x$, hence $\mathbb{\mu}\geq \mathbb{E}[\eta^{(2)}_1]$. Then (\ref{taunmain}) implies 

\begin{equation}\label{taunmain1}
\frac{\mathbb{E}[\tau_{n}]}{n}\leq \frac{\mathbb{E}[\ln n\eta^{(2)}_1]+\epsilon}{\mathbb{E}[n\eta^{(2)}_1]\mathbb{P}(0\leq X\leq \epsilon)}+\frac{1}{n}.
\end{equation}

Using (\ref{1/2}) and (\ref{1taun}), it suffices to prove that:
$$\displaystyle \lim_{n\rightarrow +\infty}\mathbb{E}[n\eta^{(2)}_1]=+\infty, \text{and} \quad \lim_{n\rightarrow +\infty}\frac{\mathbb{E}[\ln(n\eta^{(2)}_1)]}{\mathbb{E}[n\eta^{(2)}_1]}=0.$$

 It is easy to see that, using (\ref{gn1}),  there exists a positive constant $C_2$ such that $\mathbb{E}[n\eta^{(2)}_1]=\frac{n\int_{1/n}^{1}x^{-1}\Lambda(dx)}{\bar{\mu}^{(n)}}\geq C_2\frac{n\mu^{(n)}}{g_n}$, for any $n\geq 3.$ Hence  $\mathbb{E}[n\eta^{(2)}_1]$ tends to $+\infty$ since $\Lambda$ satisfies $(\ref{gnmu})$.  For the second convergence, we fix $M>e$. Then, 
  \begin{align*}
  \frac{\mathbb{E}[\ln(n\eta^{(2)}_1)]}{\mathbb{E}[n\eta^{(2)}_1]}&=\frac{\mathbb{E}[\ln(n\eta^{(2)}_1)\ind_{n\eta^{(2)}_1\geq M}]+\mathbb{E}[\ln(n\eta^{(2)}_1)\ind_{n\eta_1^{(2)}<M}]}{\mathbb{E}[n\eta^{(2)}_1]}\\
  &\leq \frac{\mathbb{E}[\ln(n\eta^{(2)}_1)\ind_{n\eta^{(2)}_1\geq M}]}{\mathbb{E}[n\eta^{(2)}_1]}+\frac{\ln M}{\mathbb{E}[n\eta^{(2)}_1]}\\
  &\leq \frac{\mathbb{E}[\ln(n\eta^{(2)}_1)\ind_{n\eta^{(2)}_1\geq M}]}{\mathbb{E}[n\eta^{(2)}_1\ind_{n\eta^{(2)}_1\geq M}]}+\frac{\ln M}{\mathbb{E}[n\eta^{(2)}_1]}\\
    &\leq  \frac{\ln M}{M}+\frac{\ln M}{\mathbb{E}[n\eta^{(2)}_1]}.\\
  \end{align*}
  The last inequality is due to the fact that for any $x\geq M>e$, we have $\frac{\ln (x)}{x}\leq \frac{\ln M}{M}$.   Since $M$ can be chosen as large as we want, then $\displaystyle \lim_{n\rightarrow +\infty} \frac{\mathbb{E}[\ln(n\eta^{(2)}_1)]}{\mathbb{E}[n\eta^{(2)}_1]}=0.$ Hence  we can conclude.
  
  \textbf{Part 2}: If $\mu=+\infty.$ We define $(\bar{\eta}^{(2)}_{i})_{i\geq 2}:=(\frac{1}{2}\ind_{\eta^{(2)}_i\geq \frac{1}{2}}+\eta^{(2)}_i\ind_{\eta^{(2)}_i<\frac{1}{2}})_{i\geq 2}$ and for $s>0,$ $\bar{\tau}_s:=\min\{i+1|\sum_{j=1}^{i}-\ln(1-\bar{\eta}^{(2)}_{j+1})\geq \ln s\eta^{(2)}_1\}$. Notice that $\mathbb{E}[-\ln(1-\bar{\eta}^{(2)}_i)]<+\infty$, then we return to the first case and get (\ref{taunmain1}) by replacing $\tau_n$ by $\bar{\tau}_n$ and keeping the same $\eta^{(2)}_1$ but with different $X$ (depending on $\bar{\eta}^{(2)}_i, i\geq 2$).  In this setting, $\mathbb{P}(0\leq X\leq \ln2)=1.$ We see that the closer $\bar{\eta}^{(2)}_i$ is to $1$, larger the $-\ln(1-\bar{\eta}^{(2)}_i)$ and hence $\tau_n\ll \bar{\tau}_n.$ Then we can conclude.
  \end{proof}
  
  \begin{rem}
  For $0<\epsilon<1$, we also have 
  \begin{equation}\label{nepsilon}
 \lim_{n\rightarrow +\infty} \sum_{i=1}^{+\infty}\mathbb{E}[\Delta^{(2)}_i(1-\Delta^{(2)}_i)^{n(1-\epsilon)}]=0.
  \end{equation}
 The proof is all the same. The only thing different is that  in place of (\ref{1taun}), we have
$  \sum_{i=1}^{+\infty}\mathbb{E}[\Delta^{(2)}_i(1-\Delta^{(2)}_i)^{n(1-\epsilon)}]\leq C\mathbb{E}[\frac{\tau_n-1}{n}+\sum_{i=\tau_n}^{+\infty}\bar{\Delta}^{(2)}_i],$
with $C$ larger than $1$ and depends on $\epsilon.$
  \end{rem}
  
Now we can start to study at first $P_{t/\mu^{(n)}}^{(2, n)}$. 

\begin{cor}
\begin{equation}
\lim_{t\rightarrow +\infty}\liminf_{n\rightarrow +\infty}P_{t/\mu^{(n)}}^{(2, n)}=1. 
\end{equation}
\end{cor}
\begin{proof} Recall that $\{e_i^{(2)}\}_{i\geq n}$ are i.i.d exponential variables with parameter $\int_0^1x^{-2}\Lambda_2(dx)=\bar{\mu}^{(n)}$, as  defined in section 2.3. Let $\tau_n(t)=\max\{j: \sum_{i=1}^je_i^{(2)} \leq t/\mu^{(n)}\}$. Then 
\begin{equation}\label{{1}}
P_{t/\mu^{(n)}}^{(2, n)}=\mathbb{E}[\sum_{i=1}^{\tau_n(t)}\Delta_i^{(2)}-\sum_{i=1}^{\tau_n(t)}\Delta_i^{(2)}(1-\Delta_i^{(2)})^{n-1}].\end{equation}

Due to Proposition \ref{plambda}, we have 
$$\displaystyle \lim_{n\rightarrow +\infty}\mathbb{E}[\sum_{i=1}^{\tau_n(t)}\Delta_i^{(2)}(1-\Delta_i^{(2)})^{n-1}]\leq \lim_{n\rightarrow +\infty}\mathbb{E}[\sum_{i=1}^{+\infty}\Delta_i^{(2)}(1-\Delta_i^{(2)})^{n-1}]=0.$$
Then it suffices to prove that 
\begin{equation}\label{taun1}
\displaystyle \lim_{t\rightarrow +\infty}\liminf_{n\rightarrow +\infty}\mathbb{E}[\sum_{i=1}^{\tau_n(t)}\Delta_i^{(2)}]=1.
\end{equation} Let $\mathcal{E}_j=\bar{\mu}^{(n)} \sum_{i=1}^je_i^{(2)}$, which is the sum of $j$ i.i.d unit exponential variables.  Let $I_{n}=\bar{\mu}^{(n)}/\mu^{(n)}$. Then 
\begin{align*}
\tau_n(t)&=\max\{j: \mathcal{E}_j\leq tI_n\}.\\
\end{align*}
For any fixed $0<\beta<1,$

\begin{align}\label{taunt2}
&\quad \mathbb{P}\left(\tau_n(t)\in [0, \beta tI_n)\bigcup (tI_n/\beta, +\infty )\right)\nonumber\nonumber\\
&=\mathbb{P}(\mathcal{E}_{\lceil \beta tI_n \rceil }\geq tI_n)+\mathbb{P}(\mathcal{E}_{\lfloor tI_n/\beta \rfloor }\leq tI_n)=o((tI_n)^{-1}),
\end{align}
where the last equality is a large deviation result (for example, see Theorem 1.4 of  \cite{den2008large}). 
).
Notice that $\tau_n(t)$ is independent of $\{\Delta_i^{(2)}\}_{i\geq 1}$, then
\begin{align*}
\mathbb{E}[\sum_{i=1}^{\tau_n(t)}\Delta_i^{(2)}]&=\mathbb{E}[1-(1-1/I_n)^{\tau_n(t)+1}]\\
&= \mathbb{E}[1-(1-1/I_n)^{\tau_n(t)+1}\ind_{tI_n\beta \leq \tau_n(t)\leq tI_n/\beta}]+o((tI_n)^{-1})\\
&\geq \mathbb{E}[1-(1-1/I_n)^{tI_n\beta}\ind_{tI_n\beta \leq \tau_n(t)\leq tI_n/\beta}]+o((tI_n)^{-1}).\\
\end{align*}
Notice that $I_n\geq 1$ and the term at the right of the above inequality satisfies
$$\lim_{t\rightarrow +\infty}\liminf_{n\rightarrow+\infty} \mathbb{E}[1-(1-1/I_n)^{tI_n\beta}\ind_{tI_n\beta \leq \tau_n(t)\leq tI_n/\beta}]+o((tI_n)^{-1})=1$$
 Then we can conclude (\ref{taun1}). 

\end{proof}
\begin{rem}
For $0<\epsilon<1$, we also have 
\begin{equation}\label{nepsilon-}
\lim_{t\rightarrow +\infty}\liminf_{n\rightarrow +\infty}P_{t/\mu^{(n)}}^{(2, \lceil n-n\epsilon \rceil)}=1. 
\end{equation}

To prove this, in the proof of this corollary, on should replace (\ref{{1}}) by
$$P_{t/\mu^{(n)}}^{(2, n)}=\mathbb{E}[\sum_{i=1}^{\tau_n(t)}\Delta_i^{(2)}-\sum_{i=1}^{\tau_n(t)}\Delta_i^{(2)}(1-\Delta_i^{(2)})^{\lceil n-n\epsilon\rceil-1}].$$

The first term satisfies (\ref{taun1}). For the second term, using (\ref{nepsilon}), we get $\displaystyle \lim_{n\rightarrow +\infty}\mathbb{E}[\sum_{i=1}^{\tau_n(t)}\Delta_i^{(2)}(1-\Delta_i^{(2)})^{\lceil n-n\epsilon\rceil-1}]=0.$ Then (\ref{nepsilon-}) is proved.

\end{rem}

The next corollary is straightforward using (\ref{1}), (\ref{1k}) and (\ref{nepsilon-}).  

\begin{cor} For any $0<\epsilon<1$, 
$$\lim_{t\rightarrow +\infty}\liminf_{n\rightarrow +\infty}P_{1,2}^{(n, \lceil n-n\epsilon\rceil, k )}(t/\mu^{(n)})=1, \lim_{t\rightarrow +\infty}\liminf_{n\rightarrow +\infty}P_{1,2}^{(n, \lceil n-n\epsilon\rceil )}(t/\mu^{(n)})=1, $$
\end{cor}

\subsection{ Proofs of main results.}

\textrm{\\}
\textbf{Proof of Theorem \ref{gnmu0}}
\begin{proof} Fix $t>0$ and $0<\epsilon<1$. Considering the measure division construction for two-type $\Lambda$-coalescents, let $\Pi$ be the path of $\Pi^{(1, n)}$ chosen at the step $0$ and define the event
$$E'=\{|\Pi(t/\mu^{(n)})|\geq n-n\epsilon\}\bigcap\{\{1\} \in \Pi(t/\mu^{(n)}) \}.$$
Recall that $\{|\Pi^{(1, n)}(t/\mu^{(n)})|\geq n-n\epsilon\}$ implies that there are at least $n-\lceil2n\epsilon \rceil$ singletons at time $t/\mu^{(n)}$.
For $n$ large enough, using the exchangeability property, we have $\mathbb{P}(E')\geq \frac{n-\lceil 2n\epsilon\rceil}{n}(1-\kappa_n(t))$, where $\kappa_n(t)=\mathbb{P}(|\Pi^{(1, n)}(t/\mu^{(n)})|< n-n\epsilon)$ and $\kappa_n(t)=o(n^{-1})$ due to the inequality (\ref{lambda1<}) . For $\epsilon$ small enough and $n$ large enough, we have $\mathbb{P}(E')$ as close as we want to $1$.  We define  another event 
$$E'':=\{\{1\} \text{ is coalesced at its \textit{first marking time} within $[0, t)$.}\}$$

Then due to (\ref{1}) and $P_t^{(2, m)}$ is increasing on $m$, we get
\begin{equation}\label{|}\mathbb{P}(E''|E')\geq P_{t/\mu^{(n)}}^{(2, \lceil n-n\epsilon \rceil)}.\end{equation}

Let $0<t_1<t$, 

\begin{align}\label{thm1}
\mathbb{P}(T_1^{(n)}\geq t_1/\mu^{(n)})&=\mathbb{P}(T_1^{(n)}\geq t_1/\mu^{(n)}, E'\bigcap E'')+\mathbb{P}(T_1^{(n)}\geq t_1/\mu^{(n)}, (E'\bigcap E'')^c)\nonumber\\
&=\mathbb{P}(L_1^{(2, n)}\geq t_1/\mu^{(n)}, E'\bigcap E'')+\mathbb{P}(T_1^{(n)}\geq t_1/\mu^{(n)}, (E'\bigcap E'')^c)
\end{align}

Corollary \ref{lll} tells that $\mathbb{P}(L_1^{(2, n)}\geq t_1/\mu^{(n)} |E')=exp(-t_1)$ and it has been proved that $\mathbb{P}(E'\cap E'')=\mathbb{P}(E')\mathbb{P}(E''|E')$ can be made as close as possible to $1$ by taking $\epsilon$ small enough and $t$ large enough and $n$ tending to $+\infty$. Hence the first term of (\ref{thm1}) can be made as close as we want to $exp(-t_1)$ and the second term is close to $0$. Then we can conclude.

\end{proof}
\textbf{Proof of Theorem \ref{kd}} 
\begin{proof}
We prove instead for $k\in\mathbb{N}$:
\begin{equation}\label{mainkd1} \mu^{(n)}(T_1^{(n)},T_2^{(n)},\cdots, T_k^{(n)}) \stackrel{(d)}{\rightarrow } (e_1,e_2,\cdots , e_k),\end{equation}
which is equivalent to (\ref{mainkd}) (see Billingsley [\cite{MR1324786}, p.19]). We will give the proof for $k=2$ and leave the easy extension to readers. The proof is similar to that of Theorem \ref{gnmu0}.  Let $\Pi$ be the path of $\Pi^{(1, n)}$ chosen at step 0. Let $t>0, 0<\epsilon<1$ and define the event 
$$F':=\{|\Pi(t/\mu^{(n)})|\geq n-n\epsilon\}\bigcap\{\{1\}, \{2\} \in \Pi(t/\mu^{(n)}) \}.$$

Using the same arguments, we get $\mathbb{P}(F')\geq \frac{{n-\lceil2n\epsilon \rceil\choose 2}}{{n\choose 2}}(1-\kappa_n(t))$. We then define the event 
$$F'':=\{\{1\}, \{2\} \text{ are both coalesced at their first \textit{marking times} within $[0, t)$.}\}$$
Then due to (\ref{1k}) and $P_t^{(2, m)}$ is increasing on $m$, we get
$$P(F''|F')\geq 1-2(1-P_{t/\mu^{(n)}}^{(2, \lceil n-n\epsilon \rceil)}),$$
which is as close as possible to $1$ for $t$ large and $n$ tending to $+\infty.$

Let $0\leq t_1,t_2\leq t$. Then
\begin{align}\label{thm}
&\quad \mathbb{P}(T_1^{(n)}\geq t_1/\mu^{(n)}, T_2^{(n)}\geq t_2/\mu^{(n)})\nonumber\\
&=\mathbb{P}(T_1^{(n)}\geq t_1/\mu^{(n)}, T_2^{(n)}\geq t_2/\mu^{(n)}, F'\bigcap F'')+\mathbb{P}(T_1^{(n)}\geq t_1/\mu^{(n)}, T_2^{(n)}\geq t_2/\mu^{(n)}, (F'\bigcap F'')^c)\nonumber\\
&=\mathbb{P}(L_1^{(2,n)}\geq t_1/\mu^{(n)}, L_2^{(2,n)}\geq t_2/\mu^{(n)}, F'\bigcap F'')+ \mathbb{P}(T_1^{(n)}\geq t_1/\mu^{(n)}, T_2^{(n)}\geq t_2/\mu^{(n)}, (F'\bigcap F'')^c).
\end{align}

As shown that $\mathbb{P}((F'\cap F''))$ can be made as close as possible to $1$ by taking $t$ large enough and $\epsilon$ small enough, tending $n$ to $+\infty$.  Then the second term in (\ref{thm}) is close to $0.$ Using Corollary \ref{lll}, the first term is as close as possible to $e^{-t_1-t_2}$ by tending $n$ to $+\infty$ with $t$ large enough. Then we can conclude.
\end{proof}

\textbf{Proof of Corollary \ref{momk}}

\begin{proof}We prove at first the case of one external branch length. We seek to prove the uniform integrability of $\{(\mu^{(n)}T_1^{(n)})^k, n\geq 2\}$ for any $k\geq 0$. One needs only to show that for any $k\in\mathbb{N}$, $\sup\{\mathbb{E}[(\mu^{(n)}T_1^{(n)})^k]| n\geq 2\}<+\infty$ (see Lemma 4.11 of \cite{MR1876169} and  Problem 14 in section 8.3  of \cite{breiman1992probability}). Let $M>0, 0<\epsilon<1$, $\beta_n=|\Pi^{(n)}(M/\mu^{(n)})|$ and $n_0=\min\{i|\mu^{(i)}>0\}$. To avoid invalid calculations, we set $\mu^{(n)}=1$ if $n<n_0$.  Using the Markov property,  we have
$$T_1^{(n)}\ll M/\mu^{(n)}+\bar{T}_1^{(\beta_n)}\ind_{T_1^{(n)}\geq M/\mu^{(n)}},$$
where $\bar{T}_1^{(n)}\stackrel{(d)}{=}T_1^{(n)}, n\geq 2$ and conditional on $\beta_n,$ $\bar{T}_1^{(\beta_n)}$ is independent of $\{\ind_{T_1^{(n)}\geq M/\mu^{(n)}}\}$. Then for $n\epsilon\geq n_0,$
\begin{align}\label{mobond}
\mathbb{E}[(\mu^{(n)}T_1^{(n)})^k]&\leq \mathbb{E}[(M+\mu^{(n)}\bar{T}_1^{(\beta_n)}\ind_{\mu^{(n)}T_1^{(n)}> M})^{k}]\leq (2M)^k+ \mathbb{E}[(2\mu^{(n)}\bar{T}_1^{(\beta_n)}\ind_{\mu^{(n)}T_1^{(n)}> M})^{k}]\nonumber\\
&\leq (2M)^{k}+(\mathbb{E}[2\mu^{(n)}\bar{T}_1^{(n)}\ind_{\beta_n=n}])^k]+\mathbb{E}[(2\mu^{(n)}\bar{T}_1^{(\beta_n)}\ind_{\mu^{(n)}T_1^{(n)}> M, n\epsilon\leq \beta_n\leq n-1})^{k}]\nonumber\\
&\quad+\mathbb{E}[(2\mu^{(n)}\bar{T}_1^{(\beta_n)}\ind_{\mu^{(n)}T_1^{(n)}> M, \beta_n<n\epsilon})^{k}]\nonumber\\
&\leq (2M)^{k}+exp(-\frac{Mg_n}{\mu^{(n)}})\mathbb{E}[(2\mu^{(n)}\bar{T}_1^{(n)})^k]\nonumber\\
&\quad+\mathbb{P}(\mu^{(n)}T_1^{(n)}> M)(2\frac{\mu^{(n)}}{\mu^{(n\epsilon)}})^{k}\max\{\mathbb{E}[(\mu^{(j)}\bar{T}_1^{(j)})^{k}]|j\in[n\epsilon, n-1]\}\nonumber\\
&\quad+\mathbb{P}(\beta_n<n\epsilon)\mathbb{E}[\frac{\beta_n}{n}(2\frac{\mu^{(n)}}{\mu^{(\beta_n)}})^k(\mu^{(\beta_n)}\bar{T}_1^{(\beta_n)})^{k}|\beta_n<n\epsilon],
\end{align}
where  $exp(-\frac{Mg_n}{\mu^{(n)}})$ in the second term at right of the last inequality is the probability for no coalescence within $[0, M/\mu^{(n)}]$. The third term is due to the fact that $\mu^{(n)}$ is an increasing function of $n$ when $n\geq n_0$. The fourth term is due to exchangeability which says that the probability for $\{1\}$ not to have coalesced at $M/\mu^{(n)}$ when there exist only $\beta_n$ blocks is less than $\frac{\beta_n}{n}.$ One needs the following three estimates to prove the boundedness of $(\mathbb{E}[(\mu^{(n)}T_1^{(n)})^k])_{n\geq 2}$. 
\begin{itemize}
\item  Estimation of $exp(-\frac{Mg_n}{\mu^{(n)}})2^k:$
Notice that for $n\geq n_0$,
$$\frac{g_n}{\mu^{(n)}}=\frac{\int_{0}^{1}(1-(1-x)^n-nx(1-x)^{n-1})x^{-2}\Lambda(dx)}{\int_{1/n}^{1}x^{-1}\Lambda(dx)}\geq \frac{\int_{1/n}^{1}(1-(1-x)^n-nx(1-x)^{n-1})x^{-2}\Lambda(dx)}{\int_{1/n}^{1}x^{-1}\Lambda(dx)}\geq \frac{e-2}{e}.$$
And if $2\leq n<n_0$, we have  $exp(-\frac{Mg_n}{\mu^{(n)}})=exp(-Mg_n)\stackrel{M\rightarrow +\infty}{\rightarrow }0.$ Hence if $M$ is large enough, we have, for any $n\geq 2,$
\begin{equation}\label{M}exp(-\frac{Mg_n}{\mu^{(n)}})2^k\leq \frac{1}{4}.
\end{equation}

\item Estimation of $\mathbb{P}(\mu^{(n)}T_1^{(n)}> M)(2\frac{\mu^{(n)}}{\mu^{(n\epsilon)}})^{k}:$
Due to Corollary \ref{maincor}, we get $\displaystyle \lim_{n\rightarrow +\infty}\frac{\mu^{(n)}}{\mu^{(n\epsilon)}}=1,$ and Theorem \ref{gnmu0} gives 
$\displaystyle \lim_{n\rightarrow +\infty}\mathbb{P}(\mu^{(n)}T_1^{(n)}> M)=exp(-M).$ Hence by taking $M$ large enough, we have  for any $n\geq 2,$

\begin{equation}\label{undemi}\mathbb{P}(\mu^{(n)}T_1^{(n)}> M)(2\frac{\mu^{(n)}}{\mu^{(n\epsilon)}})^{k}\leq \frac{1}{4}.\end{equation}

\item Estimation of $\frac{\beta_n}{n}(2\frac{\mu^{(n)}}{\mu^{(\beta_n)}})^{k}, \beta_n<n\epsilon:$
Using the notations in Proposition \ref{muequi},  for $\beta_n\geq n_0,$ we have

\begin{equation}\label{betan}\frac{\mu^{(n)}}{\mu^{(\beta_n)}}=exp(\int_{1/n}^{1/\beta_n}\frac{f(x)}{x}dx)\frac{1-f(1/n)}{1-f(1/\beta_n)}.\end{equation}
Let $n_1>n_0$ such that for any $n\geq n_1$, we have  $f(1/n)\leq \frac{1}{2k}$. Hence for any $a,b\geq n_1$, $\frac{1-f(a)}{1-f(b)}\leq 2$. This $n_1$ can be found since $f(1/n)$ tends to $0$ as $n$ tends to $+\infty.$ Then (\ref{betan}) implies, for $\beta_n\geq n_1,$

\begin{equation*}\frac{\mu^{(n)}}{\mu^{(\beta_n)}}\leq 2(\frac{n}{\beta_n})^{\frac{1}{2k}}.\end{equation*}

Hence if $n_1\leq \beta_n< n\epsilon$ and $\epsilon\leq 4^{-2k-2}$, 
\begin{equation*}
\frac{\beta_n}{n}(2\frac{\mu^{(n)}}{\mu^{(\beta_n)}})^{k}\leq 4^k(\frac{\beta_n}{n})^{1/2}< 4^k(\epsilon)^{1/2}\leq  \frac{1}{4}.
\end{equation*}
 If $\beta_n<n_1,$ due to Corollary \ref{maincor}, one could find a large number $n_2$ such that $n_2>n_1$ and for any $n\geq n_2$
 \begin{equation*}
\frac{\beta_n}{n}(2\frac{\mu^{(n)}}{\mu^{(\beta_n)}})^{k}=\frac{\beta_n}{n}(2\mu^{(n)})^{k}\leq \frac{1}{4}.
\end{equation*}

In total, if $n\geq n_2$ and $\beta_n<n\epsilon,$ then 
\begin{equation}\label{betan2}\frac{\beta_n}{n}(2\frac{\mu^{(n)}}{\mu^{(\beta_n)}})^{k}\leq \frac{1}{4}.\end{equation}
\end{itemize}

Using (\ref{mobond}), (\ref{M}), (\ref{undemi}) and (\ref{betan2}), we get

\begin{align}\label{mobond1}\mathbb{E}[(\mu^{(n)}T_1^{(n)})^k]&\leq \frac{4}{3}(2M)^{k}+\frac{1}{3}\max\{\mathbb{E}[(\mu^{(j)}\bar{T}_1^{(j)})^{k}]|j\in[n\epsilon, n-1]\}+\frac{1}{3}\mathbb{E}[(\mu^{(\beta_n)}\bar{T}_1^{(\beta_n)})^{k}|\beta_n<n\epsilon]\nonumber\\
&\leq \frac{4}{3}(2M)^{k}+\frac{2}{3}\max\{\mathbb{E}[(\mu^{(j)}\bar{T}_1^{(j)})^{k}]|j\leq n-1\}.
\end{align}
The above inequality is valid for a large $M$, $\epsilon=4^{-2k-2}$ and $n\geq n_2$. Let $C_3\geq \max\{\mathbb{E}[(\mu^{(j)}T_1^{(j)})^{k}], 4(2M)^k| 2\leq j<n_2\}$, then for any $n\geq 2,$ $C_3\geq \mathbb{E}[(\mu^{(n)}T_1^{(n)})^{k}]$ using (\ref{mobond1}). Then we can conclude.

The case of multiple external branch lengths is merely a consequence of the case of one external branch length, the Cauchy-Schwarz inequality and also a uniform integrability  argument.
\end{proof}
\textbf{Proof of Corollary \ref{ext}}
\begin{proof}
Notice that $\{T_i^{(n)}\}_{1\leq i\leq n}$ are exchangeable. Hence Corollary \ref{momk} shows that 
$$\displaystyle \lim_{n\rightarrow +\infty}\mathbb{E}[\mu^{(n)}L_{ext}^{(n)}/n]=\lim_{n\rightarrow +\infty}\mathbb{E}[\mu^{(n)}(T_1^{(n)}+T_2^{(n)}+\cdots+T_n^{(n)})/n]=\lim_{n\rightarrow +\infty}\mathbb{E}[\mu^{(n)}T_1^{(n)}]=1,$$

 and 
 
\begin{align*}
\displaystyle \lim_{n\rightarrow +\infty}\Var(\mu^{(n)}L_{ext}^{(n)}/n)&=\displaystyle \lim_{n\rightarrow +\infty}\frac{\mathbb{E}[n(\mu^{(n)}T_i^{(n)})^2]+n(n-1)\mathbb{E}[(\mu^{(n)})^2T_1^{(n)}T_2^{(n)}]-n^2(\mathbb{E}[\mu^{(n)}T_1^{(n)}])^2}{n^2}\\
&=\displaystyle \lim_{n\rightarrow +\infty}\frac{\Var(\mu^{(n)}T_1^{(n)})+(n-1)Cov(\mu^{(n)}T_1^{(n)}, \mu^{(n)}T_2^{(n)})}{n}=0.
\end{align*}

 Hence $\mu^{(n)}L_{ext}^{(n)}/n$ converges in $L^2$ to $1$.\end{proof}

Before proving Corollary \ref{total}, we study at first a problem of sensibility of a recurrence satisfied by $(T_1^{(n)})_{n\geq 2}$. More precisely, if $a_n=\mathbb{E}[T_1^{(n)}]$, then $a_n$ satisfies a recurrence (see \cite{dhersin2012external}): $a_1=0$, and for $n\geq 2$, we have

\begin{equation}\label{rec1}a_n=c_n+\sum_{k=1}^{n-1}p_{n,k}\frac{k-1}{n}a_k, \end{equation}
where $(c_n)_{n\geq 2}=(\frac{1}{g_n})_{n\geq2}$ and $p_{n,k}$ is defined in (\ref{pbk}).
Due to Corollary \ref{momk}, we have $\displaystyle \lim_{n\rightarrow +\infty}\mu^{(n)}a_n=1.$ The question is as follows: what is the limit behavior of $a_n$ if we set initially the values of $(a_i)_{1\leq i\leq n_0}$ with $n_0\geq 1$ without using (\ref{rec1}) and replace $c_n$ by $c'_n=\frac{1}{g_n}+o(\frac{1}{g_n})$? It is answered in the next lemma.

\begin{lem}\label{reclem}
Let $(a'_i)_{1\leq i\leq n_0}$ be $n_0$ real numbers and for $n>n_0$
\begin{equation}\label{rec2}a'_n=c'_n+\sum_{k=1}^{n-1}p_{n,k}\frac{k-1}{n}a'_k,\end{equation}
where $(c'_n)_{n>n_0}$ is a sequence which satisfies $c'_n=\frac{1}{g_n}+o(\frac{1}{g_n})$.
Then   $$\displaystyle \lim_{n\rightarrow +\infty}\mu^{(n)}a'_n=1.$$
\end{lem}
\begin{proof}
We fix $\epsilon>0$ and let $n_{\epsilon}> n_0$ such that $c'_n\leq \frac{1+\epsilon}{g_n}$ for $n> n_{\epsilon}$. We set $M=\max\{|a'_i|, a_i | 1\leq i\leq n_{\epsilon}\}$.

Let us at first  look at (\ref{rec1}) which has the following interpretation using random walk: A walker stands initially at point $n$, then after time $c_n$, he jumps to point $k_1$ with probability $p_{n,k_1}$, then after time $\frac{k_1-1}{n}c_{k_1}$, he jumps to $k_2$ with probability $p_{k_1, k_2}$, and then after time $\frac{(k_1-1)(k_2-1)}{nk_1}c_{k_2}$, he jumps to the next point, etc. If he falls at point $1$, then this walk is finished. It is easy to see that $a_n$ is the expectation of the total walking time. One notices that there is a scaling effect on the walking time. More precisely, let $l\geq 1$ and  $n=k_0> k_1>\cdots >k_l\geq 1$ such that the walker jumps from $k_i$ to $k_{i+1}$ for $0\leq i\leq l-1$. Then conditional on this walking history, the remaining walking time is $\left(\Pi_{i=0}^{l-1}\frac{k_{i+1}-1}{k_i}\right)a_{k_l}.$

The recurrence (\ref{rec2}) has the same interpretation. The difference is that one should stop the walker when he arrives at a point $i$ within $[1, n_0]$ and one adds a scaled value of $a'_i$ to the walking time (notice that $a'_i$ can be non-positive).  To estimate $a'_n$, we use a Markov chain $(W_i)_{i\geq 0}$ to couple the jumping structures of (\ref{rec1}) and (\ref{rec2}) : $W_0=n$, 
\begin{itemize}
\item If $W_{i}=k$ with $k\geq n_{\epsilon}$, then $W_{i+1}=k'$ with probability $p_{k,k'}$,  where $1\leq k'\leq k-1$; 
\item If $W_i<n_{\epsilon}$, then we set $W_j=W_i$ for any $j\geq i+1$. 
\end{itemize}
Notice that the jumping dynamics of both recurrences is characterized by $(W_i)_{i\geq 0}$ until arriving at a point within $[1, n_{\epsilon}]$. And we also see that $(W_i)_{i\geq 0}$ is the discrete time Markov chain related to the block counting process $|\Pi^{(n)}|$ stopped at the first time arriving within $[1, n_{\epsilon}]$.

Let $\varsigma_n=\min\{i|W_i=W_{i+1}\}$ , $C_{\varsigma_n}=\Pi_{i=0}^{\varsigma_n-1}\frac{W_{i+1}-1}{W_i}$ and $T_{\varsigma_n}$ is set to be the time to $\varsigma_n$ of the random walk related to $(\ref{rec1})$ and $T'_{\varsigma_n}$ be the corresponding time related to (\ref{rec2}).

By the scaling effect of $C_{\varsigma_n}$ on the walking time, we get 
$$a_n=\mathbb{E}[T_{\varsigma_n}+C_{\varsigma_n}a_{W_{\varsigma_n}}], a'_n=\mathbb{E}[T'_{\varsigma_n}+C_{\varsigma_n}a'_{W_{\varsigma_n}}]. $$

Due to the definitions of $M, n_{\epsilon}$, we obtain
$$a_n-M\mathbb{E}[C_{\varsigma_n}]\leq \mathbb{E}[T_{\varsigma_n}]\leq a_n; \quad a'_n-M\mathbb{E}[C_{\varsigma_n}]\leq \mathbb{E}[T'_{\varsigma_n}]\leq a'_n+M\mathbb{E}[C_{\varsigma_n}]; \quad  \mathbb{E}[T'_{\varsigma_n}]\leq (1+\epsilon)\mathbb{E}[T_{\varsigma_n}].$$

Notice that $\mathbb{E}[C_{\varsigma_n}]\leq \frac{n_{\epsilon}}{n}$ and due to Corollary \ref{maincor}, we have $\displaystyle \lim_{n\rightarrow +\infty}\frac{M\mu^{(n)}}{n}=0.$ Hence $\displaystyle \lim_{n\rightarrow +\infty}M\mathbb{E}[C_{\varsigma_n}]\mu^{(n)}=0.$ Then we can conclude that for $n$ large, $a'_n\leq (1+2\epsilon)a_n.$ In the same way, we can prove also $a'_n\geq (1-2\epsilon')a_n$ for another small positive number $\epsilon'$ with $n$ large enough. Hence we deduce the lemma.
\end{proof}

\textbf{Proof of Corollary \ref{total}}
\begin{proof}
Let $b_n=\mathbb{E}[\mu^{(n)}L_{total}^{(n)}/n]$. Then looking at the first coalescence of the process $\Pi^{(n)}$, we have,

\begin{equation}\label{mainrec}b_1=0;  b_n=\frac{\mu^{(n)}}{g_n}+\sum_{k=1}^{n-1}p_{n,k}\frac{k\mu^{(n)}}{n\mu^{(k)}}b_k, n\geq 2.\end{equation}

If for some $k$,  $\mu^{(k)}=0,$ then we set $\mu^{(k)}=1$ to avoid invalid calculations.
To use Lemma \ref{reclem}, we write (\ref{mainrec}) as:

\begin{equation}\label{mainrec1}b_1=0; b_n=\frac{\mu^{(n)}}{g_n}+\sum_{k=1}^{n-1}p_{n,k}\frac{\mu^{(n)}}{n\mu^{(k)}}b_k+\sum_{k=1}^{n-1}p_{n,k}\frac{(k-1)\mu^{(n)}}{n\mu^{(k)}}b_k, n\geq 2.\end{equation}

We at first prove that $\sum_{k=1}^{n-1}p_{n,k}\frac{\mu^{(n)}}{n\mu^{(k)}}=o(\frac{\mu^{(n)}}{g_n})$. Indeed, due to (\ref{w}),  let $a=\int_0^1(1-(1-x)^{n-1})x^{-1}\Lambda(dx)$ and $M>0$, then 

\begin{equation}\label{xxx}\mathbb{P}(X_1^{(n)}\geq Ma)\leq \frac{\mathbb{E}[X_1^{(n)}]}{Ma}\leq \frac{n}{Mg_n}.\end{equation}
Using Corollary \ref{maincor}, we have $\displaystyle \limsup_{n\rightarrow +\infty}\frac{a}{n}\leq \lim_{n\rightarrow +\infty}\frac{\int_0^{1/n}(n-1)\Lambda(dx)+\mu^{(n)}}{n}=0, \lim_{n\rightarrow +\infty}\frac{\mu^{(n)}}{\mu^{(n-Ma)}}=1.$ Then for $n$ large enough,
\begin{align*}
\sum_{k=1}^{n-1}p_{n,k}\frac{\mu^{(n)}}{n\mu^{(k)}}&=\sum_{k=1}^{\lfloor n-Ma\rfloor}p_{n,k}\frac{\mu^{(n)}}{n\mu^{(k)}}+\sum_{k=\lfloor n-Ma\rfloor+1}^{n-1}p_{n,k}\frac{\mu^{(n)}}{n\mu^{(k)}}\\
&\leq \mathbb{P}(X_1^{(n)}\geq Ma)\mathbb{E}[\frac{\mu^{(n)}}{n\mu^{(n-X_1^{(n)})}}|X_1^{(n)}\geq Ma]+\frac{\mu^{(n)}}{\mu^{(n-Ma)}n}\\
&\leq \frac{\mu^{(n)}}{Mg_n}\max\{\frac{1}{\mu^{(k)}}| 1\leq k\leq n\}+\frac{\mu^{(n)}}{\mu^{(n-Ma)}n},
\end{align*}
where the first term at the right of the the last inequality is due to (\ref{xxx}) and can be made as small as we want w.r.t $\frac{\mu^{(n)}}{g_n}$ when $M$ is large enough. Notice that $n^{-1}=o(\frac{\mu^{(n)}}{g_n})$ due to $(\ref{gnmu})$. Then the second term $\frac{\mu^{(n)}}{\mu^{(n-Ma)}n}=o(\frac{\mu^{(n)}}{g_n})$ using also $\displaystyle \lim_{n\rightarrow +\infty}\frac{\mu^{(n)}}{\mu^{(n-Ma)}}=1.$ Then $\sum_{k=1}^{n-1}p_{n,k}\frac{\mu^{(n)}}{n\mu^{(k)}}=o(\frac{\mu^{(n)}}{g_n})$.

We now only need to prove that $(b_k)_{k\geq 2}$ are bounded, since in this case, $\sum_{k=1}^{n-1}p_{n,k}\frac{\mu^{(n)}}{n\mu^{(k)}}b_k=o(\frac{\mu^{(n)}}{g_n})$ and we apply Lemma \ref{reclem} to (\ref{mainrec1}). We construct another recurrence:

\begin{equation}\label{mainrec2}b'_1=0; b'_n=\frac{C\mu^{(n)}}{g_n}+\sum_{k=1}^{n-1}p_{n,k}\frac{(k-1)\mu^{(n)}}{n\mu^{(k)}}b'_k, n\geq 2.\end{equation}
where $C$ is a positive number. If $C=1$, this is exactly a transformation of the recurrence (\ref{rec1}). Let $M'(C)=\sup\{b'_n|n\geq 1\}$. Then it is easy to see that $M'(C)=CM'(1)$. Let $n_0\geq 1$, such that for $n\geq n_0$, we have $\sum_{k=1}^{n-1}p_{n,k}\frac{\mu^{(n)}}{n\mu^{(k)}}M'(1)\leq \frac{1}{2}\frac{\mu^{(n)}}{g_n}.$ Then for $C\geq 2, n\geq n_0,$

\begin{equation}\label{part1}\frac{\mu^{(n)}}{g_n}+\sum_{k=1}^{n-1}p_{n,k}\frac{\mu^{(n)}}{n\mu^{(k)}}M'(C)\leq \frac{C\mu^{(n)}}{g_n}.\end{equation}

For $2\leq n<n_0$, we set $C$ large enough such that 

\begin{equation}\label{part2}\frac{\mu^{(n)}}{g_n}+\sum_{k=1}^{n-1}p_{n,k}\frac{\mu^{(n)}}{n\mu^{(k)}}\max\{b_i| 1\leq i< n_0\}\leq \frac{C\mu^{(n)}}{g_n}.\end{equation}

Comparing the coefficients and initial values of recurrences (\ref{mainrec1}) and (\ref{mainrec2}) using (\ref{part1}) and (\ref{part2}), we deduce that $b_n\leq b'_n\leq M'(C).$ Hence we can conclude.

\end{proof}

\textbf{Acknowledgements:}
The author benefited from the support of the "Agence Nationale de la Recherche": ANR MANEGE (ANR-09-BLAN-0215).

The author wants to thank his supervisor, Prof Jean-St\'ephane Dhersin, for fruitful discussions and
also for his careful reading of a draft version resulting in many helpful comments. The author also
wants to thank Prof Martin M\"ohle for a discussion on Theorem \ref{gnmu0}.


\begin{thebibliography}{10}

\bibitem{arnason2004mitochondrial}
E.~Arnason.
\newblock Mitochondrial cytochrome b dna variation in the high-fecundity
  atlantic cod: trans-atlantic clines and shallow gene genealogy.
\newblock {\em Genetics}, 166(4):1871--1885, 2004.

\bibitem{berestycki-2008-44}
Julien Berestycki, Nathanael Berestycki, and Jason Schweinsberg.
\newblock Small-time behavior of beta coalescents.
\newblock {\em Ann. Inst. H. Poincar{\'e} Probab. Statist.}, 44(2):214--238,
  2008.

\bibitem{MR1324786}
Patrick Billingsley.
\newblock {\em Probability and measure}.
\newblock Wiley Series in Probability and Mathematical Statistics. John Wiley
  \& Sons Inc., New York, third edition, 1995.
\newblock A Wiley-Interscience Publication.

\bibitem{MR2156553}
Michael G.~B. Blum and Olivier Fran\c{c}ois.
\newblock Minimal clade size and external branch length under the neutral
  coalescent.
\newblock {\em Adv. in Appl. Probab.}, 37(3):647--662, 2005.

\bibitem{bolthausen1998ruelle}
E.~Bolthausen and A.S. Sznitman.
\newblock On ruelle's probability cascades and an abstract cavity method.
\newblock {\em Communications in mathematical physics}, 197(2):247--276, 1998.

\bibitem{boom1994mdv}
JDG Boom, EG~Boulding, and AT~Beckenbach.
\newblock {Mitochondrial DNA variation in introduced populations of Pacific
  oyster, Crassostrea gigas, in British Columbia}.
\newblock {\em Canadian journal of fisheries and aquatic sciences(Print)},
  51(7):1608--1614, 1994.

\bibitem{breiman1992probability}
Leo Breiman.
\newblock Probability, classics in applied mathematics, vol. 7.
\newblock {\em Society for Industrial and Applied Mathematics (SIAM),
  Pennsylvania}, 1992.

\bibitem{caliebe2007length}
A.~Caliebe, R.~Neininger, M.~Krawczak, and U.~Roesler.
\newblock {On the length distribution of external branches in coalescence
  trees: genetic diversity within species}.
\newblock {\em Theoretical Population Biology}, 72(2):245--252, 2007.

\bibitem{DDS2008}
Jean-Fran\c{c}ois Delmas, Jean-St{\'e}phane Dhersin, and Arno
  Siri-J{\'e}gousse.
\newblock Asymptotic results on the length of coalescent trees.
\newblock {\em Ann. Appl. Probab.}, 18(2):997--1025, 2008.

\bibitem{den2008large}
Frank Den~Hollander.
\newblock {\em Large deviations}, volume~14.
\newblock Amer Mathematical Society, 2008.

\bibitem{dhersin2012length}
Jean-St{\'e}phane Dhersin, Arno Siri-J{\'e}gousse, Fabian Freund, and Linglong
  Yuan.
\newblock On the length of an external branch in the beta-coalescent.
\newblock {\em Stochastic Process. Appl.}, 123:1691--1715, 2013.

\bibitem{dhersin2012external}
Jean-St\'ethane Dhersin and Martin M{\"o}hle.
\newblock On the external branches of coalescent processes with multiple
  collisions with an emphasis on the bolthausen-sznitman coalescent.
\newblock {\em arXiv preprint arXiv:1209.3380}, 2012.

\bibitem{MR2353033}
Michael Drmota, Alex Iksanov, Martin Moehle, and Uwe Roesler.
\newblock Asymptotic results concerning the total branch length of the
  {B}olthausen-{S}znitman coalescent.
\newblock {\em Stochastic Process. Appl.}, 117(10):1404--1421, 2007.

\bibitem{Eldon2006}
B.~Eldon and J.~Wakeley.
\newblock Coalescent processes when the distribution of offspring number among
  individuals is highly skewed.
\newblock {\em Genetics}, 172:2621--2633, 2006.

\bibitem{MR0270403}
William Feller.
\newblock { An introduction to probability theory and its applications.
  {V}ol. {II}.}
\newblock Second edition. John Wiley \& Sons Inc., New York, 1971.

\bibitem{MR2554368}
F.~Freund and M.~M{\"o}hle.
\newblock On the time back to the most recent common ancestor and the external
  branch length of the {B}olthausen-{S}znitman coalescent.
\newblock {\em Markov Process. Related Fields}, 15(3):387--416, 2009.

\bibitem{gnedin2012asymptotics}
A.~Gnedin, A.~Iksanov, and A.~Marynych.
\newblock On asymptotics of the beta-coalescents.
\newblock {\em arXiv preprint arXiv:1203.3110}, 2012.

\bibitem{MR2484170}
Alexander Gnedin, Alex Iksanov, and Martin M{\"o}hle.
\newblock On asymptotics of exchangeable coalescents with multiple collisions.
\newblock {\em J. Appl. Probab.}, 45(4):1186--1195, 2008.

\bibitem{hedgecock19942}
D.~Hedgecock.
\newblock Does variance in reproductive success limit effective population
  sizes of marine organisms?
\newblock {\em Genetics and evolution of aquatic organisms}, page 122, 1994.

\bibitem{janson2011total}
S.~Janson and G.~Kersting.
\newblock On the total external length of the Kingman coalescent.
\newblock {\em Electronic Journal of Probability}, 16:2203--2218, 2011.

\bibitem{MR1876169}
Olav Kallenberg.
\newblock {\em Foundations of modern probability}.
\newblock Probability and its Applications (New York). Springer-Verlag, New
  York, second edition, 2002.

\bibitem{MR671034}
J.~Kingman.
\newblock The coalescent.
\newblock {\em Stochastic Process. Appl.}, 13(3):235--248, 1982.

\bibitem{MR633178}
J.~Kingman.
\newblock On the genealogy of large populations.
\newblock {\em J. Appl. Probab.}, (Special Vol. 19A):27--43, 1982.
\newblock Essays in statistical science.

\bibitem{kingman2000oc}
JFC Kingman.
\newblock {Origins of the Coalescent 1974-1982}.
\newblock {\em Genetics}, 156(4):1461--1463, 2000.

\bibitem{MR2684740}
M.~M{\"o}hle.
\newblock Asymptotic results for coalescent processes without proper
  frequencies and applications to the two-parameter {P}oisson-{D}irichlet
  coalescent.
\newblock {\em Stochastic Process. Appl.}, 120(11):2159--2173, 2010.

\bibitem{mohle2010asymptotic}
M.~M{\"o}hle.
\newblock Asymptotic results for coalescent processes without proper
  frequencies and applications to the two-parameter poisson-dirichlet
  coalescent.
\newblock {\em Stochastic Processes and their Applications},
  120(11):2159--2173, 2010.

\bibitem{MR1742892}
Jim Pitman.
\newblock Coalescents with multiple collisions.
\newblock {\em Ann. Probab.}, 27(4):1870--1902, 1999.

\bibitem{MR1742154}
Serik Sagitov.
\newblock The general coalescent with asynchronous mergers of ancestral lines.
\newblock {\em J. Appl. Probab.}, 36(4):1116--1125, 1999.

\bibitem{MR1781024}
Jason Schweinsberg.
\newblock Coalescents with simultaneous multiple collisions.
\newblock {\em Electron. J. Probab.}, 5:Paper no.\ 12, 50 pp. (electronic),
  2000.

\end{thebibliography}
\end{document}